\documentclass[final,3p,times,authoryear]{elsarticle}
\usepackage{amsthm,amsmath,amssymb,amsfonts,bbm,graphics,a4wide,rotating,graphicx}
\usepackage{color}
\usepackage{natbib}

\usepackage{slashbox}
\usepackage{multirow}
\usepackage[francais,english]{babel}

\usepackage{tabularx}
\def\1{\mbox{1\hspace{-0.15em}I}}

\newtheorem{Prop}{Proposition}
\newtheorem{Lemma}{Lemma}

\newtheorem{Rk}{Remark}
\renewenvironment{proof}{\noindent{\bf Proof.}}{\hfill
  $\blacksquare$\par\noindent}

\numberwithin{equation}{section}

\begin{document}

\begin{frontmatter}

\title{Variable selection through CART\tnoteref{t1}}
\tnotetext[t1]{The authors acknowledge the support of the  French Agence Nationale
de la Recherche  (ANR), under grant ATLAS (JCJC06\_137446) ''From Applications
to Theory in Learning and Adaptive Statistics''.}


\author[]{Marie SAUVE}

\author[N]{Christine TULEAU-MALOT\corref{cor1}}
\ead{malot@unice.fr}

\cortext[cor1]{Corresponding author. Tel.: (+33) 04 92 07 60 20; Fax: (+33) 04 93 51 79 74.}

\address[N]{Laboratoire Jean-Alexandre Dieudonn\'e, CNRS UMR 6621,\\ Universit\'e de Nice Sophia-Antipolis,\\ Parc Valrose, 06108 Nice Cedex 2, France.\vspace*{0.5cm}}

\begin{abstract}
This paper deals with variable selection in the regression and binary classification frameworks.
It proposes an automatic and exhaustive procedure which relies on the use of the CART algorithm
and on model selection via penalization.
This work, of theoretical nature, aims at determining adequate penalties, i.e. penalties which
allow to get oracle type inequalities justifying the performance of the proposed procedure. Since the exhaustive procedure can not be executed when the number of variables is too big, a more practical procedure is also proposed and still theoretically validated.
A simulation study completes the theoretical results.

\end{abstract}


\end{frontmatter}

\section{Introduction} \label{Intro}

This paper deals with variable selection in nonlinear regression and classification using CART estimation and model selection approach. Our aim is to propose a theoretical variable selection procedure for nonlinear models and to consider some practical approaches.\\

Variable selection is a very important topic since we have to consider problems where the number of variables is very large while the number of variables that are really explanatory can be much smaller. This is the reason why we are interesting in their importance. The variable importance is a notion which allows the quantification of the ability of a variable to explain the studied phenomena. The formula, for the computation, depends on the considered model. In the literature, there are many variable selection procedures which combine primarily a concept of variable importance and model estimation. If we refer to the work of Kohavi and John (\cite{Kohavi}) or Guyon and Elisseff (\cite{Guyon}), these methods are ``filter'', ``wrapper'' or ``embedded'' methods. To summarize, (i) filter method is a pre-processing step which does not depend on the learning algorithm, (ii) in the wrapper method the learning model is used to induce the final model but also to search the optimal feature subset, and (iii) for embedded methods the features selection and the learning part can not be separated.\\
Let us mention some of those methods, in the regression and/or the classification framework.\\

\subsection{General framework and State of the Art}

Let us consider a linear regression model $
Y=\sum_{j=1}^p \beta_jX^j+\varepsilon = X\beta +\varepsilon
$
where $\varepsilon$ is an unobservable noise, $Y$ the response and ${X=(X^1, \ldots, X^p)}$ a vector of $p$ explanatory variables.
Let $\{(X_i,Y_i)_{1 \leq i \leq n}\}$ be a sample, i.e. $n$ independent copies of the pair of random variables $(X,Y)$. \\

The well-known Ordinary Least Square (OLS) estimator provides an useful way to estimate the vector $\beta$ but it suffers from a main drawback: it is not adapted to variable selection since, when $p$ is large, many components of $\beta$ are non zero.
However, if OLS is not a convenient method to perform variable selection, the least squares criterion often appears in model selection.
For example, Ridge Regression and Lasso (wrapper methods) are penalized versions of OLS.
Ridge Regression (see Hastie \cite{Hastie}) involves a $L_2$ penalization
which produces the shrinkage of $\beta$ but does not force any coefficients
of $\beta$ to be zero. So, Ridge Regression is better than OLS, but it is not a
variable selection method unlike Lasso.
Lasso (see Tibshirani \cite{Tibshirani}) uses the least squares criterion
penalized by a $L_1$ penalty term. By this way, Lasso shrinks some
coefficients and puts the others to zero.
Thus, this last method performs variable selection but computationally,
its implementation needs quadratic programming techniques.\\

Penalization is not the only way to perform variable or model selection.
For example, we can cite the Subset Selection (see Hastie \cite{Hastie})
which provides, for each ${k \in \{1, \ldots, p\}}$,
the best subset of size $k$,
i.e. the subset of size $k$ which gives smallest residual sum of squares.
Then, by cross validation, the final subset is selected.
This wrapper method is exhaustive: it is consequently difficult to use it in practice
when $p$ is large. Often, Forward or Backward Stepwise Selection
(see Hastie \cite{Hastie}) are preferred since they are computationally
efficient methods. But, they may eliminate useful predictors.
Since they are not exhaustive methods they may not reach the global optimal
model. In the regression framework, there exists an efficient algorithm
developped by Furnival and Wilson (\cite{furnival}) which achieves the optimal
model, for a small number of explanatory variables,
without exploring all the models.\\

More recently, the most promising method seems to be the method called Least Angle Regression (LARS) due to Efron {\it{et al.}} (\cite{Efron}).
Let ${\mu=x\beta}$ where ${x=(X_1^T, \ldots, X_n^T)}$. LARS builds an estimate
of $\mu$ by successive steps.
It proceeds by adding, at each step, one covariate to the model, as Forward Selection.
At the begining, ${\mu=\mu_0=0}$. At the first step,
LARS finds the predictor $X^{j_1}$ most correlated with the response $Y$
and increases $\mu_0$ in the direction of $X^{j_1}$ until another predictor
$X^{j_2}$ has a larger correlation with the current residuals.
Then, $\mu_0$ is replaced by $\mu_1$. This step corresponds to the first step
of Forward Selection. But, unlike Forward Selection, LARS is based on an
equiangulary strategy. For example, at the second step,
LARS proceeds equiangulary between $X^{j_1}$ and $X^{j_2}$ until another
explanatory variable enters.
This method is computationally efficient and gives good results in practice.
However, a complete theoretical elucidation needs further investigation.\\
For linear regression, some works are also based on variable importance assessment; the aim is to produce a relative importance of regressor variables. Gr\"{o}mping (\cite{Gromping}) proposes a study of some estimators of relative importance based on variance decomposition.\\

In the context of nonlinear models, Sobol (\cite{Sobol}) proposes an extension of the notion of relative importance via the Sobol sensitivity indices, indices which take part to the sensitivity analysis (cf. Saltelli {\it{et al.}} \cite{Saltelli}). The idea of variable importance is not so recent since it can be found in the book about Classification And Regression Trees of Breiman {\it{et al.}} (\cite{Breiman}) who introduce the variable importance as the decrease of node impurity measures, or in the studies about Random Forests by Breiman {\it{et al.}} (\cite{Breiman2}, \cite{Cutler}) where the variable importance is more a permutation importance index. Thanks to this notion, the variables can be ordered and we can easily deduce some filter or wrapper methods to select some of them. But, there exists also some embedded purposes based on those notions or some others. Thus, D\`{i}az-Uriarte and Alvarez de Andr\'{e}s (\cite{Diaz}) propose the following recursive strategy. They compute the Random Forests variable importance and they delete the 20\% of variables having the smallest importance: with the remaining variables, they construct a new forest and repeat the procedure. At the end, they compare all the forest models and conserve the one having the smallest Out Of Bag error rate. Poggi and Tuleau (\cite{Poggi}) develop a method based on CART and on a stepwise ascending strategy combined with an elimination step while Genuer {\it{et al.}} (\cite{Genuer}) propose a procedure based on Random Forest combined with elimination, ranking and variable selection steps. Guyon {\it{et al.}} (\cite{Guyon2}) propose a method of selection, called SVM-RFE, utilizing Support Vector Machine methods based on Recursive Feature Elimination. Recently, this approach has been modified by Ben Ishak {\it{et al.}} (\cite{Ben}) using a stepwise strategy.\\

\subsection{Main goals}
In this paper, the purpose is to propose, for regression and classification frameworks, a variable selection procedure, based on CART, which is adaptative and theoretically validated. This second point is very important and establishes a real difference with existing works since actually most of the practical method for both frameworks are not validated because of the use of Random Forest or arbitrary thresholds on the variable importance. The method consists in applying the CART algorithm to each possible subset of variables and then considering model selection via penalization
(cf. Birg\'e and Massart \cite{Birge}),
to select the set which minimizes a penalized criterion.
In the regression and classification frameworks, we determine via oracle bounds, the expressions of this penalized criterion.\\

More precisely, let ${\mathcal{L}=\{(X_1,Y_1), \ldots , (X_n,Y_n)\}}$ be a sample, i.e. independent copies of a pair $(X,Y)$, where $X$ takes its values
in $\mathcal{X}$, for example $\mathbb{R}^p$, with distribution $\mu$
and $Y$ belongs to $\mathcal{Y}$ ($\mathcal{Y}=\mathbb{R}$ in the regression framework and $\mathcal{Y}=\{0;1\}$ in the classification one).
Let $s$ be the regression function or the Bayes classifier according to the considered framework.
We write ${X=(X^1, \ldots , X^p)}$ where the $p$ variables $X^j$,
with ${j \in \{1,2, \ldots ,p\}}$, are the explanatory variables.
We denote by $\Lambda$ the set of the $p$ explanatory variables,
i.e. ${ \Lambda=\{X^1,X^2, \ldots , X^p\} }$.
The explained variable $Y$ is called the response.
When we deal with variable selection, there exists two different objectives (cf. Genuer {\it{et al.}} \cite{Genuer}): the first one consists in determining all the important variables highly related to the response $Y$ whereas the second one is to find the smallest subset of variables to provide a good prediction of $Y$. Our purpose here is to find a subset $M$ of $\Lambda$, as small as possible, such that the variables in $M$ enable to predict the response $Y$.\\

To achieve this objective, we split the sample $\mathcal{L}$ in three subsamples $\mathcal{L}_1$, $\mathcal{L}_2$ and $\mathcal{L}_3$
of size $n_1$, $n_2$ and $n_3$ respectively.
In the following, we consider two cases: the first one is ``$\mathcal{L}_1$ independent of $\mathcal{L}_2$''
and the second corresponds to ``$\mathcal{L}_1 = \mathcal{L}_2$''.
Then we apply the CART algorithm to all the subsets of $\Lambda$ (an overview of CART is given later and for more details, the reader can refer to Breiman {\it{et al.}} \cite{Breiman}).
More precisely, for any $M \in \mathcal{P}(\Lambda)$,
we build the maximal tree by the CART growing procedure using the subsample $\mathcal{L}_1$.
This tree, denoted $T_{max}^{(M)}$, is constructed thanks to the class of admissible splits $\mathcal{S}p_M$ which involves only the variables of $M$.
For any $M \in \mathcal{P}(\Lambda)$ and any $T \preceq T_{max}^{(M)}$,
we consider the space $S_{M,T}$ of $\mathbb{L}_{\mathcal{Y}}^2(\mathbb{R}^p,\mu)$ composed by all the piecewise constant functions
with values in $\mathcal{Y}$ and defined on the partition $\tilde{T}$
associated with the leaves of $T$.
At this stage, we have the collection of models
$$ \{ S_{M,T}, \ \ M \in \mathcal{P}(\Lambda) {\text{ and }} T \preceq T_{max}^M \} $$
which depends only on $\mathcal{L}_1$.
Then, for any $(M,T)$, we denote $\hat{s}_{M,T}$ the $\mathcal{L}_2$ empirical risk minimizer on $S_{M,T}$.
$$ \hat{s}_{M,T}= \underset{u \in S_{M,T}}{argmin} \ \gamma_{n_2}(u)
\text{ with } \gamma_{n_2}(u) =
\frac{1}{n_2}\sum_{(X_i,Y_i) \in \mathcal{L}_2} \left( Y_i -u(X_i) \right)^2.$$
Finally, we select $(\widehat{M,T})$ by minimizing the penalized contrast:
$$ \widehat{(M,T)} = \underset{(M,T)}{argmin} \left\{
\gamma_{n_2}(\hat{s}_{M,T}) + pen(M,T) \right\} $$
and we denote the corresponding estimator $\tilde{s} = \hat{s}_{\widehat{M,T}}$.\\

Our purpose is to determine the penalty function $pen$ such that the model $\widehat{(M,T)}$ is close to the optimal one. This means that the model selection procedure should satisfy an oracle inequality i.e.:
$$ \mathbb{E} \left[  l(s, \tilde{s}) \ | \mathcal{L}_1 \right] \leq C
\underset{(M,T)}{\inf} \bigg\{ \mathbb{E} \left[  l(s,\hat{s}_{M,T}) \ |
\mathcal{L}_1 \right] \bigg\}, \quad {\text{ $C$ close to 1 }}$$
where $l$ denotes the loss function and $s$ the optimal predictor. The main results of this paper give adequate penalties defined up to two multiplicative constants $\alpha$ and $\beta$.
Thus we have a family of estimators $\tilde{s}(\alpha, \beta)$
among which the final estimator is chosen using the test sample $\mathcal{L}_3$. This third sub-sample is admittedly introduce for practice but we consider it also in the theoretical part since we obtain some results on it.\\

The described procedure is, of course, a theoretical one since,
when $p$ is too large, it may be impossible, in practice,
to take into account all the $2^p$ sets of variables.
A solution consists of determining, at first, few data-driven subsets of variables which are adapted to perform variable selection
and then applying our procedure to those subsets.
As this family of subsets, denoted $\mathcal{P}^*$, is constructed thanks to the data,
the theoretical penalty, determined when the procedure involves the $2^p$ sets, is still adapted for the procedure restricted to $\mathcal{P}^*$ since this subset is not deterministic.\\

The paper is organized as follows.
After this introduction, the Section \ref{Preliminaries} recalls the different steps of the CART algorithm and defines some notations.
The Sections \ref{Regression} and \ref{Classification} present the results obtained in the regression and classification frameworks. In both sections, the results have the same spirit, however since the frameworks differ, the assumptions and the penalty functions are different. This is the reason why, for clarity, we divide our results. In the Section \ref{Simulation}, we apply our procedure to a simulated example and we compare the results of the procedure when, on the one hand, we consider all sets of variables and, on the other hand, we take into account only a subset determined thanks to the Variable Importance defined by Breiman {\it{et al.}} (\cite{Breiman}). Sections \ref{Appendix} and \ref{Proofs} collect lemmas and proofs.

\section{Preliminaries} \label{Preliminaries}

\subsection{Overview of CART and variable selection}

In the regression and classification frameworks and thanks to a
training set, CART splits recursively the observations space $\mathcal{X}$ and
defines a piecewise constant function on this partition which is
called a predictor or a classifier according to the case. CART proceeds
in three steps: the construction of a maximal tree, the construction of nested
models by pruning and a final model selection. In the following, we give a brief summary; for details, readers may refer to the seminal book of Breiman {\it{et al.}} (\cite{Breiman}) or to Gey's vulgarization articles (\cite{Nedelec2}, \cite{Gey}). \\

The first step consists of the construction of a nested
sequence of partitions of $\mathcal{X}$ using binary
splits. A useful representation of this
construction is a tree composed of nonterminal and terminal nodes. To each nonterminal node is associated a binary split which is just a question of the form $\mathbf{(X^j \leq c_j)}$ for numerical variables or $\mathbf{(X^j \in S_j)}$ for qualitative ones. Such split involves only one original explanatory variable and is
determined by maximizing a quality criterion derived from impurity function. For instance, in the regression framework the criterion for a node $\mathbf{t}$ is the decrease of $\mathbf{R(t)}$ where $\mathbf{R(t)=\frac{1}{n} \sum_{X_i \in t} (Y_i-\bar{Y}(t))^2}$ with $\mathbf{\bar{Y}(t)}$ the arithmetical mean of $\mathbf{Y}$ over $\mathbf{t}$. This is just the estimate of the error. In the classification framework, the criterion is the decrease in impurity which is often the Gini index $\mathbf{i(t)=\sum_{i \neq j} p(i|t)p(j|t)}$ with $\mathbf{p(i|t)}$ the posterior probability of the class $\mathbf{i}$ in $\mathbf{t}$. In this case, the criterion is less intuitive but the estimate of the misclassification rate has too many drawbacks to be used. The tree associated to the finest partition, that is to say the one with one observation or observations with the same response by element, is the maximal tree. This tree is too complex and too faithful with the training sample. This is the reason of the next step.\\
The principle of the pruning step is to extract, from the maximal tree a
sequence of nested subtrees which minimize a penalized criterion. This penalized criterion, proposed by Breiman {\it{et al.}} (\cite{Breiman}) realizes a tradeoff between the goodness of fit and the complexity
of the tree (or model) measured by the number of leaves.\\
At last, via a test sample or cross validation, a subtree is selected
among the preceding sequence.\\

CART is an algorithm which builds a binary decision tree.
A first idea is to perform variable selection by retaining the variables appearing in the tree.
This has many drawbacks since on the one hand, the number of selected variables may be too large, and on the other hand, some really important variables could be hidden by the selected ones.\\
A second approach is based on the Variable Importance (VI) introduced by Breiman {\it{et al.}} (\cite{Breiman}).
This criterion, calculated with respect to a given tree (typically coming from the procedure CART),
quantifies the contribution of each variable by awarding it a note between $0$ and $100$.
The variable selection consists of keeping the variables whose notes are greater than an arbitrary threshold.
But, there is, at present, no way to automatically determine the threshold and such a method does not allow to suppress highly dependent influent variables.\\
In this paper, we propose another approach which consists of applying CART to each subset of variables and choosing the set which minimizes an adequate penalized criterion.

\subsection{The context}

The paper deals with two frameworks: the regression and the binary classification.
In both cases, we denote
\begin{eqnarray}
\label{s}
s=\underset{u:\mathbb{R}^p \rightarrow \mathcal{Y}} {argmin} \
\mathbb{E}\left[ \gamma(u,(X,Y)) \right] {\text{ with }}
\gamma(u,(x,y))=(y-u(x))^2.
\end{eqnarray}

The quantity $s$ represents the best predictor according to the quadratic contrast $\gamma$.
Since the distribution $P$ is unknown, $s$ is unknown too.
Thus, in the regression and classification frameworks, we use
$(X_1,Y_1),...,(X_n,Y_n)$, independent copies of $(X,Y)$, to construct
an estimator of $s$. The quality of this one is measured by the loss
function $l$ defined by:
\begin{eqnarray}
\label{l}
l(s,u)=\mathbb{E}[\gamma(u,.)]-\mathbb{E}[\gamma(s,.)].
\end{eqnarray}

In the regression case,
the expression of $s$ defined in (\ref{s}) is
\begin{eqnarray*}
\forall x \in \mathbb{R}^p, \ \ s(x)=\mathbb{E}[Y \vert X=x],
\end{eqnarray*}
and the loss function $l$ given by (\ref{l}) is the
$\mathbb{L}^2(\mathbb{R}^p,\mu)$-norm, denoted $ \left\|.\right\|_\mu
$.\\
In this context, each $(X_i,Y_i)$ satisfies
$$ Y_i=s(X_i)+\varepsilon_i $$
where $(\varepsilon_1,...,\varepsilon_n)$ is a sample such that
$\mathbb{E}\left[ \varepsilon_i | X_i \right] = 0$.
In the following, we assume that the variables $\varepsilon_i$ have
exponential moments around $0$ conditionally to $X_i$. As explained in (\cite{Sauve}),
this assumption can be expressed by the existence of two constants
$\sigma \in \mathbb{R}_+^*$ and $\rho \in \mathbb{R}_+$ such that
\begin{eqnarray} \label{A}
\text{for any } \lambda \in \left( -1/ \rho, 1/ \rho \right) \text{, }
\log \mathbb{E} \left[ e^{\lambda \varepsilon_i} \big| X_i \right] \leq
\frac{\sigma^2 \lambda^2}{2\left( 1-\rho | \lambda | \right).}
\end{eqnarray}
$\sigma^2$ is necessarily greater than $\mathbb{E}(\varepsilon_i^2)$
and can be chosen as close to $\mathbb{E}(\varepsilon_i^2)$ as we want,
but at the price of a larger $\rho$.

\begin{Rk}
If $\rho = 0$ in (\ref{A}),
the random variables $\varepsilon_i$ are said to be sub-Gaussian
conditionally to $X_i$.\\
\end{Rk}

In the classification case, the Bayes classifier $s$, given by (\ref{s}),
is defined by:
\begin{eqnarray*}
\forall x \in \mathbb{R}^p, \ \ s(x)=\1_{\eta(x) \geq 1/2} {\text{ with }} \eta(x)=\mathbb{E}[Y | X=x].
\end{eqnarray*}
As $Y$ and the predictors $u$ take their values in $\{0;1\}$, we have ${\gamma(u,(x,y)) =  \1_{u(x) \neq y}}$ so we deduce that the loss function $l$ can be expressed as:
$$l(s,u) = \mathbb{P}(Y \neq u(X))-\mathbb{P}(Y \neq s(X))
= \mathbb{E}\left[|s(X)-u(X)||2\eta(X)-1|\right].$$

For both frameworks, we consider two situations:
\begin{itemize}
\item
$(M1)$: the training sample $\mathcal{L}$ is divided in three
independent parts $\mathcal{L}_1$, $\mathcal{L}_2$ and $\mathcal{L}_3$
of size $n_1$, $n_2$ and $n_3$ respectively.
The subsample $\mathcal{L}_1$ is used to construct the maximal tree, $\mathcal{L}_2$ to prune it and $\mathcal{L}_3$ to perform the final selection;

\item
$(M2)$: the training sample $\mathcal{L}$ is divided only in two
independent parts $\mathcal{L}_1$ and $\mathcal{L}_3$.
The first one is both for the construction of the maximal tree and its pruning
whereas the second one is for the final selection.
\end{itemize}

The $(M1)$ situation is theoretically easier since all the subsamples are
independent, thus each step of the CART algorithm is performed on independent data sets. With real data, it is often difficult to split the sample in three parts because of the small number of data. That is the reason why we also consider
the more realistic situation $(M2)$.

\section{Classification} \label{Classification}

This section deals with the binary classification framework.
In this context, we know that the best predictor is the Bayes
classifier $s$ defined by:
$$ \forall x \in \mathbb{R}^p, \ \ s(x)= \1_{\eta(x) \geq 1/2} $$

A problem appears when $\eta(x)$ is close to $1/2$, because
in this case, the choice between the label $0$ and $1$ is
difficult. If $\mathbb{P}(\eta(x)=1/2) \neq 0 $, then the accuracy of the Bayes
classifier is not really good and the comparison with $s$ is not
relevant. For this reason, we consider the margin condition introduced by
Tsybakov (\cite{Tsybakov}):
$$ \exists h>0, {\text{ such  that }} \forall x \in \mathbb{R}^p, \  \ \vert 2
\eta(x) - 1 \vert \geq h. $$
For details about this margin condition, we refer to Massart (\cite{Flour}). Otherwise in (\cite{Arlot}) some considerations about margin-adaptive model selection could be found more precisely in the case of nested models and with the use of the margin condition introduced by Mammen and Tsybakov (\cite{Mammen}).\\

The following subsection gives results on the
variable selection
for the methods $(M1)$ and $(M2)$ under margin condition. More precisely, we define convenient penalty functions
which lead to oracle bounds.
The last subsection deals with the final selection by test sample $\mathcal{L}_3$.

\subsection{Variable selection via $(M1)$ and $(M2)$}

\underline{$ \bullet \ (M1)$ case :}\\
Given the collection of models
$$ \left\{ S_{M,T}, \  M \in \mathcal{P}(\Lambda) \text{ and } T \preceq T_{max}^{(M)} \right\} $$
built on $\mathcal{L}_1$, we use the second subsample $\mathcal{L}_2$ to select a model $\widehat{(M,T)}$ which is close to the optimal one.
To do this, we minimize a penalized criterion
$$ crit(M,T)= \gamma_{n_2}\left( \hat{s}_{M,T} \right) + pen\left( M,T \right)$$
The following proposition gives a penalty function $pen$ for which
the risk of the penalized estimator $\tilde{s} = \hat{s}_{\widehat{M,T}}$
can be compared to the oracle accuracy.

\begin{Prop} \label{propc1}
Let consider a penalty function of the form:
$ \forall \ M \in \mathcal{P}(\Lambda) $ and $\forall \ T \preceq
T_{max}^{(M)}$
$$ pen(M,T)=\alpha \frac{\vert T \vert}{n_2h} +
\beta \frac{\vert M \vert}{n_2h}\left(1+ \log\left(\frac{p}{\vert M \vert}\right)\right).$$
If $\alpha > \alpha_0$ and $\beta > \beta_0$, then there exists
two positive constants $C_1 >1$ and $C_2$, which only depend on $\alpha$ and $\beta$,
such that:
$$ \mathbb{E} \bigg[ l \left( s,\tilde{s} \right) | \mathcal{L}_{1} \bigg]
\leq  C_{1} \underset{(M,T)}{\inf} \bigg\{ l \left( s,S_{M,T} \right)+ pen \left( M,T \right) \bigg\}
+ C_{2}\frac{1}{n_2h} $$
where  $l(s,S_{M,T})=\underset{u \in S_{M,T}}{\inf}l(s,u)$.
\end{Prop}

The penalty is the sum of two terms. The first one is proportional to $\frac{\vert T \vert}{n_2}$ and corresponds to the penalty proposed by breiamn {\it{et al.}} (\cite{Breiman}) in their pruning algorithm. The other one is proportional to ${ \frac{\vert M
  \vert}{n_2} \left(1+ \log\left( \frac{p}{\vert M \vert} \right)\right) }$ and is due to the variable selection. It penalizes models that are based on too much explanatory variables.
For a given value of $\vert M \vert $, this result validates the CART pruning algorithm in the binary classification framework, result proved also by Gey (\cite{Gey}) in a more general situation since the author consider a less stronger margin condition.

Thanks to this penalty function, the problem can be divided in practice in
two steps:
\begin{enumerate}
\item[-]
First, for every set of variables $M$, we select a subtree
$\hat{T}_M$ of $T_{max}^{(M)}$ by
$$\hat{T}_M =\underset{T \preceq
T_{max}^{(M)}}{argmin} \left\{  \gamma_{n_2} (\hat{s}_{M,T}) +
\alpha^{\prime}  \frac{|T|}{n_2} \right\}.$$
This means that $\hat{T}_M$ is a tree obtained by the CART pruning procedure using the subsample $\mathcal{L}_2$

\item[-]
Then we choose a set $\hat{M}$ by minimizing a criterion which penalizes the big sets of variables:
$$ \hat{M} = \underset{M \in \mathcal{P}(\Lambda)}{argmin}  \left\{ \gamma_{n_2} (\hat{s}_{M,\hat{T}_M}) + pen(M,\hat{T}_M)  \right\} .$$
\end{enumerate}

The $(M1)$ situation permits to work conditionally to the construction of the maximal trees $T_{max}^{(M)}$
and to select a model among a deterministic collection.
Finding a convenient penalty to select a model among a deterministic collection is easier, but we have not always enough observations to split the training sample $\mathcal{L}$ in three subsamples. This is the reason why we study now the $(M2)$ situation.\\

\underline{$\bullet \ (M2)$ case :}\\
We manage to extend our result
for only one subsample $\mathcal{L}_1$.
But, while in the $(M1)$ method we work with the expected loss,
here we need the expected loss conditionally to ${ \{X_i, \ (X_i,Y_i) \in \mathcal{L}_1 \} }$ defined by:
\begin{eqnarray} \label{exploss}
l_1(s,u)=\mathbb{P} \left( u(X) \neq Y \vert \{X_i, \ (X_i,Y_i) \in
\mathcal{L}_1 \} \right) - \mathbb{P} \left( s(X) \neq Y \vert \{X_i,
\ (X_i,Y_i) \in \mathcal{L}_1 \} \right).
\end{eqnarray}

\begin{Prop} \label{propc2}
Let consider a penalty function of the form:
$ \forall \ M \in \mathcal{P}(\Lambda) $ and $\forall \ T \preceq
T_{max}^{(M)}$
$$ pen(M,T)  = \alpha \left[1+ \left( \vert M \vert +1 \right)
\left( 1+ \log \left( \frac {n_1}{\vert M \vert +1} \right) \right) \right]
\frac{\vert T \vert}{n_{1}h}
+ \beta  \frac{ \vert M \vert}{n_{1}h}\left(1 + \log \left( \frac{p}{\vert M \vert} \right) \right). $$
If $\alpha > \alpha_0$ and $\beta > \beta_0$, then there exists
three positive constants $C_1 >2$, $C_2$, $\Sigma$ which only depend on $\alpha$ and $\beta$,
such that, with probability $\geq 1- e^{-\xi}{\Sigma}^2 $:
$$
l_1(s,\tilde{s})  \le  C_{1}\underset{(M,T)}{\inf} \bigg\{ l_1 \left(
s,S_{M,T} \right) + pen_{n}(M,T) \bigg\} +
\frac{C_2}{n_1h}\left(1+\xi \right)
$$
where \ $l_1(s,S_{M,T})=\underset{u \in S_{M,T}}{\inf}l_1(s,u)$. \\
\end{Prop}

When we consider the $(M2)$ situation instead of the $(M1)$ one, we obtain only an inequality with high probability instead of a result in expectation, Indeed, since all the results are obtained conditionally to the construction of the maximal tree, in this second situation, it is impossible to integrate with respect to $\mathcal{L}_1$ whereas in the first situation, we integrated with respect to $\mathcal{L}_2$.\\

Since the penalized criterion depends on two parameters $\alpha$ and $\beta$, we obtain a family of predictors $\tilde{s}=\hat{s}_{\widehat{M,T}}$ indexed by $\alpha$ and $\beta$, and the associated family of sets of variables $\hat{M}$.
Now, we choose the final predictor using test sample and we deduce the corresponding set of selected variables.

\subsection{Final selection}

Now, we have a collection of predictors
$$ \mathcal{G} = \left\{ \tilde{s}(\alpha,\beta); \ \alpha>\alpha_0 \text{ and
} \beta > \beta_0 \right\} $$
which depends on $\mathcal{L}_1$ and $\mathcal{L}_2$. \\

For any $M$ of $ \mathcal{P}\left( \Lambda \right)$, the set $\left\{ T \preceq T_{max}^{(M)} \right\}$ is finite.
As $\mathcal{P}\left( \Lambda \right)$ is finite too, the cardinal $\mathcal{K}$ of $\mathcal{G}$ is finite and
$$ \mathcal{K} \leq \sum_{M \in \mathcal{P}(\Lambda)} \mathcal{K}_M $$
where $\mathcal{K}_M$ is the number of subtrees of $T_{max}^{(M)}$ obtained
by the pruning algorithm defined by Breiman {\it et al.} (\cite{Breiman}).
$\mathcal{K}_M$ is very smaller than $\left| \left\{ T \preceq T_{max}^{(M)} \right\} \right|$.
Given the subsample $\mathcal{L}_3$,
we choose the final estimator $\tilde{\tilde{s}}$ by minimizing
the empirical contrast $\gamma_{n_3}$ on $\mathcal{G}$.
$$ \tilde{\tilde{s}} = \underset{\tilde{s}(\alpha,\beta) \in
\mathcal{G}}{argmin} \ \gamma_{n_3}\left( \tilde{s}(\alpha,\beta) \right) $$

The next result validates the final selection for the $(M1)$ method.

\begin{Prop} \label{propcf}
For any $\eta \in (0,1)$, we have:
$$ \mathbb{E} \bigg[ l \left( s,\tilde{\tilde{s}} \right) \ |
\mathcal{L}_1, \mathcal{L}_2 \bigg] \leq
\frac{1+ \eta}{1- \eta} \underset{ (\alpha,\beta) }{\inf}
\bigg\{ l \left( s,\tilde{s}(\alpha,\beta) \right) \bigg\}
+ \frac{ \left( \frac{1}{3}+\frac{1}{\eta} \right) \frac{1}{1-\eta} }{n_3 h}
\log \mathcal{K} + \frac{ \frac{ 2\eta+\frac{1}{3}+\frac{1}{\eta} }{1-\eta} }{n_3 h}. $$
\end{Prop}

For the $(M2)$ method, we get exactly the same result except that the loss $l$ is replaced by the conditional loss $l_1$ (\ref{exploss}).\\

For the $(M1)$ method, since the results in expectation of the {\bf{Propositions \ref{propc1}}} and {\bf{\ref{propcf}}} involve the same expected loss, we can compare the final estimator $\tilde{\tilde{s}}$ with the entire collection of models:
$$
\mathbb{E} \bigg[ l(s,\tilde{\tilde{s}}) \ | \mathcal{L}_1,
\mathcal{L}_2 \bigg] \leq \tilde{C}_1 \underset{(M,T)}{\inf} \ \bigg\{
l(s,S_{M,T}) + pen(M,T) \bigg\} + \frac{C_2}{n_2h}+ \frac{C_3}{n_3h}\bigg(1+ \log \mathcal{K} \bigg).
$$

In the classification framework, it may be possible to obtain sharper upper bounds by considering for instance the version of  Talagrand concentration inequality developed by Rio (\cite{Rio}), or another margin condition as the one proposed by Koltchinskii (see \cite{Kolt}) and used by Gey (\cite{Gey}). However, the idea remains the same and those improvement do not have a real interest since we do not get in our work precise calibration of the constants.

\section{Regression} \label{Regression}

Let us consider the regression framework where the $\varepsilon_i$
are supposed to have exponential moments around $0$ conditionally to $X_i$ (cf. \ref{A}).

In this section, we add a stop-splitting rule in the CART growing procedure.
During the construction of the maximal trees $T_{max}^{(M)}$,
$M \in \mathcal{P}(\Lambda)$, a node is split only if the
two resulting nodes contain, at least, $N_{min}$ observations. \\

As in the classification section, the following subsection gives results on the
variable selection
for the methods $(M1)$ and $(M2)$ and the last subsection deals with the final selection by test sample $\mathcal{L}_3$.

\subsection{Variable selection via $(M1)$ and $(M2)$} \label{M1}
In this subsection, we show that for convenient constants $\alpha$ and $\beta$, the same form of penalty function as in classification framework leads to an oracle bound.\\

\underline{$\bullet \ (M1)$ case :}

\begin{Prop} \label{propM1}
Let suppose that $\|s\|_{\infty} \leq R$, with $R$ a positive constant. \\
Let consider a penalty function of the form:
$ \forall \ M \in \mathcal{P}(\Lambda) $ and $\forall \ T \preceq
T_{max}^{(M)}$
$$ pen(M,T) = \alpha \left( \sigma^2 + \rho R \right) \frac{|T|}{n_2}
+ \beta \left( \sigma^2 + \rho R \right) \frac{|M|}{n_2} \left(1+ \log\left(\frac{p}{|M|}\right)\right). $$
If $ p \leq \log n_2 $, $N_{min} \geq 24 \frac{\rho^2}{\sigma^2} \log n_2$,
$ \alpha > \alpha_0 $ and $ \beta > \beta_0 $, \\
then there exists two positive constants $C_1 >2$ and $C_2$,
which only depend on $\alpha$ and $ \beta$, such that:
\begin{eqnarray*}
\mathbb{E} \left[ \left\| s - \tilde{s} \right\|_{n_2}^2 \ |
\mathcal{L}_1 \right] &\leq &
C_1 \underset{(M,T)}{\inf} \left\{ \underset{u \in S_{M,T}}{\inf}
\left\| s-u \right\|_{\mu}^2 + pen(M,T) \right\}
+C_2 \frac{(\sigma^2 + \rho R)}{n_2} \\
& & + C(\rho, \sigma, R)
\frac{\1_{\rho \neq 0}}{n_2 \log n_2}
\end{eqnarray*}
where $\| \ . \ \|_{n_2}$ denotes the empirical norm on $\{ X_i; \ (X_i,Y_i) \in \mathcal{L}_2 \}$ and $C(\rho, \sigma, R)$ is a constant which only depends on $\rho$, $\sigma$ and $R$.
\end{Prop}

As in classification, the penalty function is the sum of two terms: one is proportional to $\frac{|T|}{n_2}$ and the other to $\frac{|M|}{n_2} \left(1+ \log\left(\frac{p}{|M|}\right)\right)$. The first term corresponds also to the penalty proposed by Breiman {\it et al.} (\cite{Breiman})
in their pruning algorithm and
validated by Gey and N\'ed\'elec (\cite{Nedelec2}) for the Gaussian regression case. This proposition validates the CART pruning penalty in a more general regression framework than the Gaussian one.\\

\begin{Rk}
In practice, since $\sigma^2$, $\rho$ and $R$ are unknown, we consider penalties of the form
$$ pen(M,T) = \alpha^{\prime} \frac{|T|}{n_2} + \beta^{\prime} \frac{|M|}{n_2}
\left(1+ \log\left(\frac{p}{|M|}\right)\right) $$
If $\rho = 0$, the form of the penalty is
$$ pen(M,T) = \alpha \sigma^2 \frac{|T|}{n_2} + \beta \sigma^2 \frac{|M|}{n_2} \left(1+ \log\left(\frac{p}{|M|}\right)\right), $$
the oracle bound becomes
$$
 \mathbb{E} \left[ \left\| s - \tilde{s} \right\|_{n_2}^2 \ | \mathcal{L}_1
\right] \leq C_1 \underset{(M,T)}{\inf} \left\{ \underset{u \in S_{M,T}}{\inf} \left\| s-u \right\|_{\mu}^2 + pen(M,T) \right\} +C_2 \frac{\sigma^2}{n_2},
$$
and the assumptions on $\|s\|_{\infty}$, $p$ and $N_{min}$ are
no longer required. Moreover, the constants $\alpha_0$ and $\beta_0$ can be taken
as follows:
 $$\alpha_0 = 2(1+3 \log 2) \ {\text{ and }} \ \beta_0 = 3.$$
 In this case $\sigma^2$ is the single unknown parameter which appears in the penalty. Instead of using $\alpha^{\prime}$ and $\beta^{\prime}$ as proposed above, we can in practice replace $\sigma^2$ by an estimator.
\end{Rk}

\underline{$\bullet \ (M2)$ case :}\\
In this situation, the same subsample $\mathcal{L}_1$ is used to build
the collection of models
$$ \left\{ S_{M,T}, \  M \in \mathcal{P}(\Lambda) \text{ and } T \preceq T_{max}^{(M)} \right\}$$
and to select one of them.\\
For technical reasons, we introduce the collection of models
$$ \left\{ S_{M,T}, \  M \in \mathcal{P}(\Lambda) \text{ and } T \in \mathcal{M}_{n_1,M} \right\} $$
where $\mathcal{M}_{n_1,M} $ is the set of trees built on the grid
$\left\{ X_i; \ \left( X_i, Y_i \right) \in \mathcal{L}_1 \right\}$
with splits on the variables in $M$.
This collection contains the preceding one and only depends on
$\left\{ X_i; \ \left( X_i, Y_i \right) \in \mathcal{L}_1 \right\}$.
We find nearly the same result as in the $(M1)$ situation.

\begin{Prop} \label{propM2}
Let suppose that $\|s\|_{\infty} \leq R$, with $R$ a positive constant. \\
Let consider a penalty function of the form:
$ \forall \ M \in \mathcal{P}(\Lambda) $ and $\forall \ T \preceq
T_{max}^{(M)}$
\begin{eqnarray*}
pen(M,T) & = & \alpha \left( \sigma^2 \left( 1+\frac{\rho^4}{\sigma^4} \log^2\left( \frac{n_1}{p} \right)  \right)
+ \rho R \right)
\left( 1 + \left( |M|+1 \right) \left( 1+ \log \left( \frac{n_1}{|M|+1} \right)  \right) \right)
\frac{|T|}{n_1} \\
& & + \beta \left( \sigma^2 \left( 1+\frac{\rho^4}{\sigma^4} \log^2\left( \frac{n_1}{p} \right)  \right) + \rho R \right) \frac{|M|}{n_1} \left(1+ \log \left(\frac{p}{|M|}\right)\right).
\end{eqnarray*}
If $ p \leq \log n_1 $, $ \alpha > \alpha_0 $ and $ \beta > \beta_0 $, \\
then there exists three positive constants $C_1 >2$, $C_2$ and $\Sigma$
which only depend on $\alpha$ and $ \beta$, such that:\\
$\forall \xi >0$, with probability $\geq 1 - e^{-\xi} \Sigma - \frac{c}{n_1 \log n_1}\1_{\rho \neq 0}$,
\begin{eqnarray*}
\left\| s - \tilde{s} \right\|_{n_1}^2 & \leq &
C_1 \underset{(M,T)}{\inf} \left\{ \underset{u \in S_{M,T}}{\inf}
\left\| s-u \right\|_{n_1}^2 + pen(M,T) \right\} \\
& & + \frac{C_2}{n_1} \left( \left( 1+\frac{\rho^4}{\sigma^4}\log^2\left( \frac{n_1}{p} \right)   \right) \sigma^2
+ \rho R \right) \xi
\end{eqnarray*}
where $\| \ . \ \|_{n_1}$ denotes the empirical norm on $\{ X_i; \ (X_i,Y_i) \in \mathcal{L}_1 \}$ and $c$ is a constant which depends on $\rho$ and $\sigma$.
\end{Prop}

Like in the $(M1)$ case, for a given $|M|$, we find a penalty
proportional to $\frac{|T|}{n_1}$ as
proposed by Breiman {\it et al.} and validated by Gey and N\'ed\'elec
in the Gaussian regression framework. So here again, we validate
the CART pruning penalty in a more general regression framework.\\
Unlike the $(M1)$ case, the multiplicative factor of $\frac{|T|}{n_1}$, in the penalty function, depends on $M$ and $n_1$. Moreover, in the method $(M2)$, the inequality is obtained only with high probability.

\begin{Rk}
If $\rho=0$, the form of the penalty is
$$ pen(M,T) = \alpha \sigma^2 \left[1 + (|M|+1)\left( 1 + \log\left( \frac{n_1}{|M|+1} \right) \right) \right] \frac{|T|}{n_1} + \beta \sigma^2 \frac{|M|}{n_1} \left(1+ \log\left(\frac{p}{|M|}\right)\right), $$
 the oracle bound is
$\forall \ \xi > 0$, with probability $\geq 1-e^{- \xi}\Sigma $,
$$ \left\| \tilde{s} - s \right\|_{n_1}^2 \leq C_1 \underset{(M,T)}{\inf} \left\{ \underset{u \in S_{M,T}}{\inf} \left\| s-u \right\|_{n_1}^2 + pen(M,T)
\right\} + C_2 \frac{\sigma^2}{n_1} \xi $$
 and the assumptions on $\|s\|_{\infty}$ and $p$ are no longer
required.
Moreover, we see that we can take
$\alpha_0 = \beta_0 = 3$.
\end{Rk}

\subsection{Final selection}

The next result validates this selection.

\begin{Prop} \label{propF}
\begin{itemize}
\item
In the $(M1)$ situation, taking $p\leq \log n_2$ and ${N_{min} \geq
4 \frac{\sigma^2 + \rho R}{R^2} \log n_2}$, we have: \\
for any $\xi >0$, with probability
$\geq 1 - e^{-\xi} - \1_{\rho \neq 0} \frac{R^2}{2(\sigma^2+ \rho R)} \frac{1}{n_2^{1- \log 2}}$,
$\forall \eta \in (0,1)$,

\begin{eqnarray*}
\left\| s - \tilde{\tilde{s}} \right\|_{n_3}^2 & \leq &
\frac{(1+\eta^{-1}-\eta)}{\eta^2}
\underset{\tilde{s}(\alpha,\beta) \in \mathcal{G}}{\inf}
\left\| s - \tilde{s}(\alpha,\beta) \right\|_{n_3}^2 \\
& & + \frac{1}{\eta^2} \left( \frac{2}{1-\eta}\sigma^2 + 8 \rho R \right)
\frac{(2 \log \mathcal{K} + \xi)}{n_3}.
\end{eqnarray*}

\item
In the $(M2)$ situation, denoting
$\epsilon(n_1) = 2 \1_{\rho \neq 0} n_1 \exp \left( - \frac{9 \rho^2 \log^2 n_1}{2(\sigma^2 + 3 \rho^2 \log n_1)} \right) $,
we have: \\
for any $\xi >0$, with probability
$\geq 1 - e^{-\xi} - \epsilon(n_1)$, $\forall \eta \in (0,1)$,

\begin{eqnarray*}
\left\| s - \tilde{\tilde{s}} \right\|_{n_3}^2 & \leq &
\frac{(1+\eta^{-1}-\eta)}{\eta^2}
\underset{\tilde{s}(\alpha,\beta) \in \mathcal{G}}{\inf}
\left\| s - \tilde{s}(\alpha,\beta) \right\|_{n_3}^2 \\
& & + \frac{1}{\eta^2} \left( \frac{2}{1-\eta}\sigma^2 + 4 \rho R
+ 12 \rho^2 \log n_1 \right)
\frac{(2 \log \mathcal{K} + \xi)}{n_3}.
\end{eqnarray*}
\end{itemize}
\end{Prop}

\begin{Rk}
If $\rho=0$, by integrating with respect to $\xi$,
we get for the two methods $(M1)$ and $(M2)$ that: \\
for any $\eta \in (0,1)$,
\begin{eqnarray*}
\mathbb{E}\left[ \left\| s - \tilde{\tilde{s}} \right\|_{n_3}^2 \ \big| \mathcal{L}_1, \ \mathcal{L}_2 \right] & \leq & \frac{1+\eta^{-1}-\eta}{\eta^2}
\underset{\tilde{s}(\alpha,\beta) \in \mathcal{G}}{\inf} \left\{  \mathbb{E}\left[ \left\| s - \tilde{s}(\alpha,\beta) \right\|_{n_3}^2 \ \big| \mathcal{L}_1, \ \mathcal{L}_2 \right] \right\} \\
& & +
 \frac{2}{\eta^2(1-\eta)} \frac{\sigma^2}{n_3} \left(2 \log \mathcal{K} +1\right).
\end{eqnarray*}

The conditional risk of the final estimator $\tilde{\tilde{s}}$
with respect to $\| \ \|_{n_3}$ is controlled by the minimum of the errors made by $ \tilde{s}(\alpha,\beta) $. Thus the test sample selection does not
alterate so much the accuracy of the final estimator.
Now we can conclude that theoretically our procedure is valid.
\end{Rk}

Unlike the classification framework, we are not able, even when $\rho=0$, to compare the final estimator $\tilde{\tilde{s}}$ with the entire collection of models since the different inequalities involve empirical norms that can not be compared.

\section{Simulations} \label{Simulation}

The aim of this section is twice. On the one hand, we illustrate by an example the theoretical procedure, described in the {\bf{Section \ref{Intro}}}.
On the other hand, we compare the results of the theoretical procedure with those obtained when we consider the procedure restricted to a family $\mathcal{P}^*$ constructed thanks to Breiman's Variable Importance.\\

The simulated example, also used by Breiman {\it{et al.}} (see \cite{Breiman} p. 237), is composed of $p=10$ explanatory variables $X^1, \ldots , X^{10}$ such that:
$$
\left\{
\begin{array}{l}
\ \ \mathbb{P}(X^1=-1)=\mathbb{P}(X^1=1)=\frac{1}{2} \\
\ \ \forall i \in \{2, \ldots , 10\}, \ \ \mathbb{P}(X^i=-1)=\mathbb{P}(X^i=0)=\mathbb{P}(X^i=1)=\frac{1}{3}
\end{array}
\right.
$$
and of the explained variable $Y$ given by:
\begin{eqnarray*}
Y=s(X^1, \ldots, X^{10})+\varepsilon =
\begin{cases}
\ \ 3+3X^2+2X^3+X^4+\varepsilon & {\text{ if   $ \quad X^1=1$ }},\\
\ \ -3+3X^5+2X^6+X^7+\varepsilon & {\text{ if   $ \quad X^1=-1$ }}.
\end{cases}
\end{eqnarray*}
where the unobservable random variable $\varepsilon$ is independent of $X^1, \ldots , X^{10}$ and normally distributed with mean $0$ and variance $2$.\\

The variables $X^8$, $X^9$ and $X^{10}$ do not appear in the definition of the explained variable $Y$, they can be considered as observable noise.\\

The {\bf{Table \ref{Tabi}}} contains the Breiman's Variable Importance.
The first row presents the explanatory variables ordered from the most influential to the less influential, whereas the second one contains the Breiman's Variable Importance Ranking.\\

\begin{table}[!h]
\begin{center}
\begin{tabular}{|c|cccccccccc|}
\hline
Variable & $X^1$ & $X^2$ & $X^5$ & $X^3$ & $X^6$ & $X^4$ & $X^7$ & $X^8$ & $X^9$ & $X^{10}$ \\
\hline
 Rank    &   1   &   2   &   3   &   5   &   4   &   7   &   6  &    8   &   9    &   10  \\
\hline
\end{tabular}
\caption{\label{Tabi} Variable Importance Ranking for the considered simulated example.}
\end{center}
\end{table}

We note that the Variable Importance Ranking is consistent with the simulated model since the two orders coincide. In fact, in the model, the variables $X^3$ and $X^6$ (respectively $X^4$ and $X^7$) have the same effect on the response variable $Y$.\\

To make in use our procedure, we consider a training sample $\mathcal{L}$ which consists of the realization of $1000$ independent copies of the pair of random variables $(X,Y)$ where $X=(X^1, \ldots , X^{10})$.\\

The first results are related to the behaviour of the set of variables associated with the estimator $\tilde{s}$.
More precisely, for given values of the parameters $\alpha$ and $\beta$ of the penalty function, we look at the selected set of variables.\\

According to the model definition and the Variable Importance Ranking, the expected results are the following ones:

\begin{itemize}
\item the size of the selected set should belong to $\{1,3,5,7,10\}$. As the variables $X^2$ and $X^5$ (respectively $X^3$ and $X^6$, $X^4$ and $X^7$ or $X^8$, $X^9$ and $X^{10}$) have the same effect on the response variable, the other sizes could not appear, theoretically;
\item the set of size $k$, $k \in \{1,3,5,7,10\}$, should contain the $k$ most important variables since Variable Importance Ranking and model definition coincide;
\item the final selected set should be $\{1,2,5,3,6,4,7\}$.
\end{itemize}

The behaviour of the set associated with the estimator $\tilde{s}$, when we apply the theoretical procedure, is summarized by the {\bf{Table \ref{Tab1}}}.\\
At the intersection of the row $\beta$ and the column $\alpha$ appears the set of variables associated with $\tilde{s}(\alpha, \beta)$.\\

\begin{table}[!h]
\begin{center}
\begin{tabular}{|c|c|c|c|c|c|c|}
\hline
\backslashbox{$\beta$}{$\alpha$} & $ \alpha \leq 0.05$ & $0.05 < \alpha \leq 0.1$ & $0.1 < \alpha \leq 2$ & $2 < \alpha \leq 12$ & $12 < \alpha \leq 60$ & $60 \leq \alpha$ \\
\hline
\multirow{2}{3cm}{\centerline{$\beta \leq 100$}} & \multirow{2}{2cm}{\centerline{$\{1,2,5,6,3,$} \centerline{$7,4,8,9,10\}$}} & \multirow{2}{2cm}{\centerline{$\{1,2,5,6,$} \centerline{$3,7,4\}$}} & \multirow{2}{2cm}{\centerline{$\{1,2,5,6,$} \centerline{$3,7,4\}$}} & \multirow{2}{2cm}{\centerline{$\{1,2,5,$} \centerline{$6,3\}$}} & \multirow{2}{2cm}{\centerline{$\{1,2,5\}$}} & \multirow{2}{1cm}{\centerline{$\{1\}$}} \\
 & & & & & & \\
\hline
\multirow{2}{3cm}{\centerline{$100 < \beta \leq 700$}} & \multirow{2}{2cm}{\centerline{$\{1,2,5,6,$} \centerline{$3,7,4\}$}} & \multirow{2}{2cm}{\centerline{$\{1,2,5,6,$} \centerline{$3,7,4\}$}} & \multirow{2}{2cm}{\centerline{$\{1,2,5,$} \centerline{$6,3\}$}} & \multirow{2}{2cm}{\centerline{$\{1,2,5,$} \centerline{$6,3\}$}} & \multirow{2}{2cm}{\centerline{$\{1,2,5\}$}} & \multirow{2}{1cm}{\centerline{$\{1\}$}} \\
 & & & & & & \\
 \hline
\multirow{2}{3cm}{\centerline{$700 < \beta \leq 1300$}} & \multirow{2}{2cm}{\centerline{$\{1,2,5,$} \centerline{$6,3\}$}} & \multirow{2}{2cm}{\centerline{$\{1,2,5,$} \centerline{$6,3\}$}} & \multirow{2}{2cm}{\centerline{$\{1,2,5,$} \centerline{$6,3\}$}} & \multirow{2}{2cm}{\centerline{$\{1,2,5,$} \centerline{$6,3\}$}} & \multirow{2}{2cm}{\centerline{$\{1,2,5\}$}} & \multirow{2}{1cm}{\centerline{$\{1\}$}} \\
 & & & & & & \\
\hline
\multirow{2}{3cm}{\centerline{$1300 < \beta \leq 1700$}} & \multirow{2}{2cm}{\centerline{$\{1,2,5\}$}} & \multirow{2}{2cm}{\centerline{$\{1,2,5\}$}} & \multirow{2}{2cm}{\centerline{$\{1,2,5\}$}} & \multirow{2}{2cm}{\centerline{$\{1,2,5\}$}} & \multirow{2}{2cm}{\centerline{$\{1\}$}} & \multirow{2}{1cm}{\centerline{$\{1\}$}} \\
 & & & & & & \\
\hline
\multirow{2}{3cm}{\centerline{$1900 <\beta$}} & \multirow{2}{2cm}{\centerline{$\{1\}$}} & \multirow{2}{2cm}{\centerline{$\{1\}$}} & \multirow{2}{2cm}{\centerline{$\{1\}$}} & \multirow{2}{2cm}{\centerline{$\{1\}$}} & \multirow{2}{2cm}{\centerline{$\{1\}$}} & \multirow{2}{1cm}{\centerline{$\{1\}$}} \\
 & & & & & & \\
\hline
\end{tabular}
\caption{\label{Tab1} In this table appears the set associated with the estimator $\tilde{s}$ for some values of the parameters $\alpha$ and $\beta$ which appear in the penalty function $pen$.}
\end{center}
\end{table}

First, we notice that those results are the expected ones.
Then, we see that for a fixed value of the parameter $\alpha$ (respectively $\beta$), the increasing of $\beta$ (resp. $\alpha$) results in the decreasing of the size of the selected set, as expected. Therefore, this decreasing is related to Breiman's Variable Importance since the explanatory variables disappear according to the Variable Importance Ranking (see {\bf{Table \ref{Tabi}}}).
As the expected final set $\{1,2,5,3,6,4,7\}$ appears in the {\bf{Table \ref{Tab1}}}, obviously, the final step of the procedure, whose results are given by the {\bf{Table \ref{Tab2}}}, returns the ``good'' set.\\

\begin{table}[!h]
\begin{center}
\begin{tabular}{|c|c|c|}
\hline
$\hat{\alpha}$ & $\hat{\beta}$ & selected set \\
\hline
0.3 & $\rightarrow$ 100 & $\{1,2,3,4,5,6,7\}$\\
\hline
\end{tabular}
\caption{\label{Tab2} In this table, we see the results of the final model selection.}
\end{center}
\end{table}

The {\bf{Table \ref{Tab2}}} provides some other informations.
At present, we do not know how to choose the parameters $\alpha$ and $\beta$ of the penalty function. This is the reason why the theoretical procedure includes a final selection by test sample. But, if we are able to determine, thanks to the data, the value of those parameters, this final step would disappear.
If we analyse the {\bf{Table \ref{Tab2}}}, we see that the ``best'' parameter $\hat{\alpha}$ takes only one value and that $\hat{\beta}$ belongs to a ``small'' range. So, those results lead to the conclusion that a data-driven determination of the parameters $\alpha$ and $\beta$ of the penalty function may be possible and that further investigations are needed.\\

As the theoretical procedure is validated on the simulated example, we consider now a more realistic procedure when the number of explanatory variables is large. It involves a smaller family $\mathcal{P}^*$ of sets of variables. To determine this family, we use an idea introduced by Poggi and Tuleau in (\cite{Poggi}) which associates Forward Selection and variable importance (VI) and whose principle is the following one.
The sets of $\mathcal{P}^*$ are constructed by invoking and testing the explanatory variables according to Breiman's Variable Importance ranking.
More precisely, the first set is composed of the most important variable according to VI.
To construct the second one, we consider the two most important variables
and we test if the addition of the second most important variable has a
significant incremental influence on the response variable.
If the influence is significant, the second set of $\mathcal{P}^*$ is composed
 of the two most importance variables.
If not, we drop the second most important variable and we consider the first
and the third most important variables and so on.
So, at each step, we add an explanatory variable to the preceding set which is less important than the preceding ones.\\

For the simulated example, the corresponding family $\mathcal{P}^*$ is:
$$
\mathcal{P}^*=\bigg\{ \{1\};\{1,2\};\{1,2,5\};\{1,2,5,6\};\{1,2,5,6,3\};\{1,2,5,6,3,7\};\{1,2,5,6,3,7,4\} \bigg\}
$$
In this family, the variables $X^8$, $X^9$ and $X^{10}$ do not appear. This is consistent with the model definition and Breiman's VI ranking.\\

The first advantage of this family $\mathcal{P}^*$ is that it involves, at the most $p$ sets of variables instead of $2^p$. The second one is that, if we perform our procedure restricted to the family $\mathcal{P}^*$, we obtain nearly the same results for the behavior of the set associated with $\tilde{s}$ than the one obtained with all the $2^p$ sets of variables (see {\bf{Table \ref{Tab1}}}). The only difference is that, since $\mathcal{P}^*$ does not contain the set of size $10$, in the {\bf{Table \ref{Tab1}}}, the set $\{1,2,3,4,5,6,7,8,9,10\}$ is replaced by $\{1,2,5,6,3,7,4\}$.

\section{Appendix} \label{Appendix}

This section presents some lemmas which are useful in the proofs of the propositions of the Sections \ref{Regression} and \ref{Classification}. The lemmas 1 to 4 are known results. We just give the statements and references for the proofs.
The lemma 5 is a variation of lemma 4.
The remaining lemmas are intermediate results which we prove to obtain both the propositions and their proofs.\\

The lemma \ref{Talagrand} is a concentration inequality due to Talagrand. This type of inequality allows to know how a random variable behaves around its expectation.

\begin{Lemma}[Talagrand]
\label{Talagrand}
Consider $n$ independent random variables $\xi_1,...,\xi_n$ with
values in some measurable space $\Theta$.
Let $\mathcal{F}$ be some countable family of real valued measurable functions on $\Theta$,
such that $\| f \|_{\infty} \leq b < \infty$ for every $f \in \mathcal{F}$. \\
Let
${ Z= \underset{f \in \mathcal{F}}{\sup} \left\vert
\sum_{i=1}^n \left( f(\xi_{i})-\mathbb{E}\left[ f(\xi_{i}) \right] \right)
\right\vert }$
and
${ \sigma^{2} = \underset{f \in \mathcal{F}}{\sup}
\left\{ \sum_{i=1}^n Var[f(\xi_{i})] \right\} }$ \\
Then, there exists $K_1$ and $K_2$ two universal constants such that
for any positive real number $x$,
$$ \mathbb{P}\left(Z \ge K_{1} \mathbb{E}[Z] + K_{2} \left(\sigma \sqrt {2x} +
bx\right)\right) \leq \exp (-x).$$
\end{Lemma}

\begin{proof}
see Massart (\cite{Toulouse}).
\end{proof}

The lemma \ref{Maximal} allows to pass from local maximal
inequalities to a global one.

\begin{Lemma}[Maximal inequality]
\label{Maximal}
Let $(\mathcal{S},d)$ be some countable set.\\
Let Z be some process indexed by $\mathcal{S}$ such that
$\underset{t \in B(u,\sigma)}{\sup} \vert Z(t)-Z(u) \vert$ has finite
expectation for any positive real $\sigma$,
with ${ B(u,\sigma)= \bigg\{t \in \mathcal{S} {\text{ such  that }} d(t,u) \le \sigma \bigg\}}$.\\
Then, for all $ \Phi : \mathbb{R} \rightarrow \mathbb{R}^{+} $ such that:
\begin{align}
-  \quad & x \rightarrow \frac {\Phi (x)}{x} {\text{ is non increasing, }}\nonumber \\
- \quad  & \forall \sigma \geq {\sigma}_{*} \quad  \mathbb{E}\left[
 \underset{t \in B(u,\sigma)}{\sup} \vert Z(t)-Z(u) \vert \right] \le \Phi(\sigma),\nonumber
\end{align}
\\
we have:
$$
\forall x \geq \sigma_* \hspace{0.5cm}
\mathbb{E}\left[ \underset{t \in \mathcal{S}}{\sup} \frac { \vert  Z(t)-Z(u)
\vert} {d^{2}(t,u)+x^{2}} \right] \le \frac {4}{x^{2}}\Phi (x).
$$
\end{Lemma}

\begin{proof}
see Massart and N\'ed\'elec (\cite{Nedelec}), section: ``Appendix: Maximal inequalities'', lemma 5.5.
\end{proof}

Thanks to the lemma \ref{Boucheron}, we see that the Hold-Out is an adaptative selection procedure for classification.

\begin{Lemma}[Hold-Out] \label{Boucheron}
Assume that we observe $N+n$ independent random variables with common
distribution $P$ depending on some parameter $s$ to be estimated.
The first $N$ observations ${ X^{\prime}=(X_1^{\prime}, \ldots, X_N^{\prime}) }$ are used to build some preliminary collection of estimators $(\hat{s}_m)_{m \in \mathcal{M}}$
and we use the remaining observations $(X_1, \ldots, X_n)$ to select some estimator $\hat{s}_{\hat{m}}$ among the
collection defined before by minimizing the empirical contrast.\\
Suppose that $\mathcal{M}$ is finite with cardinal $K$.\\
If there exists a function $w$ such that:
\begin{align}
- \quad & w : \mathbb{R}^+ \rightarrow \mathbb{R}^+, \nonumber \\
- \quad & x \rightarrow  \frac {w (x)}{x} \quad \text{is non increasing,} \nonumber \\
- \quad & \forall \epsilon > 0, \quad \underset{l(s,t) \leq \epsilon^2}{\sup} Var_P \left( \gamma(t,.)- \gamma(s,.)\right) \leq w^2(\epsilon) \nonumber
\end{align}
Then, for all $ \theta \in (0,1)$, one has:
$$
\left(1-\theta \right) \mathbb{E}\left[ l(s,\hat{s}_{\hat{m}}) \vert X^{\prime} \right] \leq \left(1+\theta \right) \underset{m \in \mathcal{M}}{\inf}l(s,\hat{s}_m)+
\delta ^2_*\left(2 \theta + \left(1+ \log K \right)\left(\frac{1}{3}+\frac{1}{\theta}\right)\right)
$$
where $\delta^2_*$ satisfies to $\sqrt{n} \delta^2_* = w(\delta_*)$.
\end{Lemma}

\begin{proof}
see (\cite{Flour}), Chapter: ``Statistical Learning'', Section: ``Advanced model selection problems''.
\end{proof}

The lemmas \ref{faux-chi2} and \ref{faux-chi2-2} are concentration inequalities for a sum of squared random variables whose Laplace transform are controlled.
The lemma \ref{faux-chi2} is due to Sauv\'e (\cite{Sauve}) and allows to generalize the model selection result of Birg\'e and Massart (\cite{Birge}) for histogram models without assuming the observations to be Gaussian.
In the first lemma, we consider only partitions $m$ of
$\left\{ 1, \ldots , n \right\}$ constructed from an initial partition $m_0$
(i.e. for any element $J$ of $m$, $J$ is the union of elements of $m_0$),
whereas in the second lemma we consider all partitions $m$ of
$\left\{ 1, \ldots , n \right\}$.

\begin{Lemma}
\label{faux-chi2}
Let $\varepsilon_1, \ldots , \varepsilon_n$ n independent and identically distributed random variables satisfying:
$$ \mathbb{E}[\varepsilon_i]=0 \ {\text{ and }} \
\text{for any } \lambda \in \left( -1/ \rho, 1/ \rho \right) \text{, }
\log \mathbb{E} \left[ e^{\lambda \varepsilon_i} \right] \leq
\frac{\sigma^2 \lambda^2}{2\left( 1-\rho | \lambda | \right)}
$$
Let $m_0$ a partition of $\left\{ 1, \ldots , n \right\}$ such that,
$\forall J \in m_0$, $|J| \geq N_{min}$. \\
We consider the collection $\mathcal{M}$ of all partitions of
$\left\{ 1, \ldots , n \right\}$ constructed from $m_0$ and the statistics
$$\chi_m^2 = \sum_{J \in m}
\frac{\left( \sum_{i \in J} \varepsilon_i \right)^2}{|J|}, \ m \in \mathcal{M}$$
Let $ \delta > 0$ and denote $ \Omega_{\delta} =
\left\{ \forall J \in m_0; \ \left| \sum_{i \in J} \varepsilon_i \right|
\leq \delta \sigma^2 |J| \right\}. $\\
Then for any $m \in \mathcal{M}$ and any $x>0,$
\begin{eqnarray*}
\mathbb{P} \left( \chi_m^2 \1_{\Omega_{\delta}} \geq
\sigma^2 |m| + 4 \sigma^2 (1+ \rho\delta)\sqrt{2|m|x} + 2 \sigma^2(1+ \rho\delta)x
\right) \leq e^{-x}
\end{eqnarray*}
and
\begin{eqnarray*}
\mathbb{P} \left( \Omega_{\delta}^c \right) \leq 2 \frac{n}{N_{min}}
\exp \left( \frac{-\delta^2 \sigma^2 N_{min}}{2(1+ \rho\delta)} \right).
\end{eqnarray*}
\end{Lemma}

\begin{proof}
\noindent Let $m \in \mathcal{M}$ and denote, for any $J \in m$,
$$
Z_J=\frac{(\sum_{i \in J}\varepsilon_i)^2}{|J|} \wedge (\delta^2 \sigma^4 |J|)
$$

\noindent $(Z_J)_{J \in m}$ are independent random variables. After calculating their moments, we deduce from Bernstein inequality that, for any $x>0$,
$$
\mathbb{P}\left( \sum_{J \in m} Z_J \geq \sigma^2|m|+4\sigma^2(1+b\delta)\sqrt{2|m|x}+2\sigma^2(1+b\delta)x\right) \leq e^{-x}
$$

\noindent
As $\sum_{J\in m} Z_J=\chi_m^2$ on the set $\Omega_{\delta}$, we get that for any $x>0$,
$$
\mathbb{P}\left( \chi_m^2 \1_{\Omega_{\delta}} \geq \sigma^2|m|+4\sigma^2(1+b\delta)\sqrt{2|m|x}+2\sigma^2(1+b\delta)x\right) \leq e^{-x}
$$

\noindent Thanks to the assumption on the Laplace transform of the $\varepsilon_i$, we have for any $J\in m_0$
$$
\mathbb{P}\left(\Big| \sum_{i \in J} \varepsilon_i \Big| \geq \delta \sigma^2 |J| \right) \leq 2exp\left(\frac{-\delta^2 \sigma^2 |J|}{2(1+b\delta)}\right)
$$

\noindent As $|J| \geq N_{min}$, we obtain
$$
\mathbb{P}(\Omega_{\delta}^c) \leq 2\frac{n}{N_{min}}exp(\frac{-\delta^2 \sigma^2 N_{min}}{2(1+b\delta)}
$$
\end{proof}

\begin{Lemma} \label{faux-chi2-2}
Let $\varepsilon_1, \ldots , \varepsilon_n$ n independent and identically distributed random variables satisfying:
$$ \mathbb{E}[\varepsilon_i]=0 \ {\text{ and }} \
\text{for any } \lambda \in \left( -1/ \rho, 1/ \rho \right) \text{, }
\log \mathbb{E} \left[ e^{\lambda \varepsilon_i} \right] \leq
\frac{\sigma^2 \lambda^2}{2\left( 1-\rho | \lambda | \right)}
$$
We consider the collection $\mathcal{M}$ of all partitions of
$\left\{ 1, \ldots , n \right\}$ and the statistics
$$\chi_m^2 = \sum_{J \in m}
\frac{\left( \sum_{i \in J} \varepsilon_i \right)^2}{|J|}, \ m \in \mathcal{M}$$
Let $ \delta > 0$ and denote $ \Omega_{\delta} =
\left\{ \forall 1 \leq i \leq n; \ \left|  \varepsilon_i \right|
\leq \delta \sigma^2 \right\}. $\\
Then for any $m \in \mathcal{M}$ and any $x>0$,
\begin{eqnarray*}
\mathbb{P} \left( \chi_m^2 \1_{\Omega_{\delta}} \geq
\sigma^2 |m| + 4 \sigma^2 (1+ \rho\delta)\sqrt{2|m|x} + 2 \sigma^2(1+ \rho\delta)x
\right) \leq e^{-x}
\end{eqnarray*}
and
\begin{eqnarray*}
\mathbb{P} \left( \Omega_{\delta}^c \right) \leq 2 n
\exp \left( \frac{-\delta^2 \sigma^2 }{2(1+ \rho\delta)} \right).
\end{eqnarray*}
\end{Lemma}

\begin{proof}
The proof is exactly the same as the preceding one. The only difference is that the set $\Omega_{\delta}$ is smaller and $N_{min} = 1$.
\end{proof}

The lemmas \ref{lemme1-poids} and \ref{lemme2-poids} give the expression
of the weights needed in the model selection procedure.

\begin{Lemma} \label{lemme1-poids}
The weights $x_{M,T} =  a|T| + b|M|\left(1 + \log\left(\frac{p}{|M|}\right)\right) $, with $a > 2 \log 2$ and $b > 1$ two absolute constants, satisfy
\begin{eqnarray} \label{cond1-poids}
\sum_{M \in \mathcal{P}(\Lambda)} \sum_{T \preceq T_{max}^{(M)}} e^{-x_{M,T}} \leq \Sigma(a,b)
\end{eqnarray}
with $\Sigma(a,b) = -\log\left( 1-e^{-(a-2 \log 2)} \right) \frac{e^{-(b-1)}}{1-e^{-(b-1)}} \in \mathbb{R}_+^* $.
\end{Lemma}

\begin{proof}
\noindent We are looking for weights $x_{M,T}$ such that the sum
$$\Sigma (\mathcal{L}_1) = \sum_{M \in \mathcal{P}(\Lambda)} \sum_{T \preceq T_{max}^{(M)}} e^{-x_{M,T}} $$
is lower than an absolute constant.

\noindent Taking $x$ as a function of the number of variables $|M|$ and of the number of leaves $|T|$, we have
\begin{eqnarray*}
\Sigma (\mathcal{L}_1) & = &  \sum_{k=1}^{p} \sum_{\underset{|M|=k}{M \in \mathcal{P}(\Lambda)}} \sum_{D=1}^{n_1} \left| \left\{ T \preceq T_{max}^{(M)}; \ |T|=D \right\} \right| \ e^{-x(k,D)}.
\end{eqnarray*}

\noindent Since
$$\left| \left\{ T \preceq T_{max}^{(M)}; \ |T|=D \right\} \right| \leq \frac{1}{D}
\begin{pmatrix}
2(D-1) \\
D-1
\end{pmatrix}
\leq \frac{2^{2D}}{D}
\ ,
$$

\noindent we get
\begin{eqnarray*}
\Sigma (\mathcal{L}_1)  & \leq &  \sum_{k=1}^{p} \left(\frac{ep}{k}\right)^k \sum_{D \geq 1} \frac{1}{D} e^{-\left( x(k,D) - (2 \log 2)D \right)}.
\end{eqnarray*}

\noindent Taking $  x(k,D)= aD + bk\left(1 + \log\left(\frac{p}{k}\right)\right) $ with $a > 2 \log 2$ and $b > 1$ two absolute constants, we have
\begin{eqnarray*}
\Sigma (\mathcal{L}_1) \leq
\left( \sum_{k \geq 1}  e^{-(b - 1)k} \right) \left( \sum_{D \geq 1} \frac{1}{D} e^{-\left( a - (2 \log 2) \right)D} \right) = \Sigma(a,b).
\end{eqnarray*}

\noindent Thus the weights $x_{M,T} = a |T| + b|M| \left(1 + \log\left(\frac{p}{|M|}\right)\right) $, with $a > 2 \log 2$ and $b > 1$ two absolute constants,
satisfy (\ref{cond1-poids}).
\end{proof}

\begin{Lemma} \label{lemme2-poids}
The weights $$ x_{M,T} = \left( a + (|M|+1)\left(1+ \log\left(\frac{n_1}{|M|+1}\right) \right) \right) |T| + b \left(1 + \log\left(\frac{p}{|M|}\right)\right)|M| $$
with $a > 0$ and $b > 1$ two absolute constants, satisfy
\begin{eqnarray} \label{cond2-poids}
\sum_{M \in \mathcal{P}(\Lambda)} \sum_{T \in \mathcal{M}_{n_1,M}} e^{-x_{M,T}} \leq \Sigma^{'}(a,b)
\end{eqnarray}
with $ \Sigma^{'}(a,b) = \frac{e^{-a}}{1-e^{-a}} \frac{e^{-(b-1)}}{1-e^{-(b-1)} }  $ and $\mathcal{M}_{n_1,M}$ the set of trees built on the grid $\{ X_i; \ (X_i,Y_i) \in \mathcal{L}_1 \}$ with splits on the variables in $M$.
\end{Lemma}

\begin{proof}
\noindent We are looking for weights $x_{M,T}$ such that the sum
$$
\sum(\{X_i; \ (X_i,Y_i) \in \mathcal{L}_1\})=\sum_{M \in \mathcal{P}(\Lambda)} \sum_{T \in \mathcal{M}_{n_1,M}} e^{-x_{M,T}}
$$
\noindent is lower than an absolute constant.

\noindent Taking $x$ as a function of the number of variables $|M|$ and the number of leaves $|T|$, we have
$$
\sum(\{X_i; \ (X_i,Y_i) \in \mathcal{L}_1\})=\sum_{k=1}^p\sum_{M \in \mathcal{P}(\Lambda) |M|=k} \sum_{D=1}^{n_1} |\{T \in \mathcal{M}_{n_1,M}; \ |T|=D\}|e^{-x_{k,D}}
$$

\noindent Since the Vapnik-Chervonenkis dimension of $\mathcal{S}p_M$ (the class of admissible splits which involves only the variables of $M$) is $|M|+1$, it follows from the lemma 2 in (\cite{Nedelec2}) that
$$
|\{T \in \mathcal{M}_{n_1,M}; \ |T|=D \}| \leq \left( \frac{n_1 e}{|M|+1} \right)^{D(|M|+1)}
$$

\noindent We get
\begin{eqnarray}
\sum (\{X_i; \ (X_i,Y_i) \in \mathcal{L}_1\}) & \leq & \sum_{k=1}^p  \binom{k}{p} \sum_{D \geq 1}e^{D[(k+1)(1+log(\frac{n_1}{k+1}))]-x(k,D)} \\
& \leq & \sum_{k=1}^p  \left(\frac{ep}{k}\right)^k \sum_{D \geq 1}e^{D[(k+1)(1+log(\frac{n_1}{k+1}))]-x(k,D)}
\end{eqnarray}

\noindent Taking $x(k,D)=D[a+(k+1)\left(1+log\left(\frac{n_1}{k+1}\right)\right)]+\alpha(k)$ with $a>0$ an absolute constant, we have
$$
\sum(\{X_i; \ (X_i,Y_i) \in \mathcal{L}_1\})  \leq  \left( \sum_{k=1}^p e^{-\left(\alpha(k)-k\left(1+log\left(\frac{p}{k}\right)\right)\right)}\right)\left(\sum_{D \geq 1}e^{-aD}\right)$$

\noindent Thus taking $x(k,D)=D[a+(k+1)\left(1+log\left(\frac{n_1}{k+1}\right)\right)]+bk\left(1+log\left(\frac{p}{k}\right)\right)$ with $a>0$ and $b>1$ two absolute constants, we have
$$
\sum(\{X_i; \ (X_i,Y_i) \in \mathcal{L}_1\})  \leq  \left( \sum_{k \geq 1}e^{-(b-1)k} \right)\left( \sum_{D \geq 1}e^{-aD}\right)=\Sigma^{\prime}(a,b)
$$
\noindent Thus the weights $x(M,T)=|T|[a+(|M|+1)\left(1+log\left(\frac{n_1}{|M|+1}\right)\right)]+b|M|\left(1+log\left(\frac{p}{|M|}\right)\right)$ with $a>0$ and $b>1$ two absolute constants, satisfy (\ref{cond2-poids}).
\end{proof}

The two last lemmas provide controls in expectation for processes studied in classification.

\begin{Lemma} \label{Nedelec}
Let $(X_1,Y_1),...,(X_n,Y_n)$ be $n$ independent observations taking their values
in some measurable space ${\Theta \times \{0,1\}}$, with common distribution $P$.
We denote $d$ the $L^2(\mu)$ distance where $\mu$ is the marginal distribution of $X_i$.\\
Let
$S_T$ the set of piecewise constant functions defined on the partition $\tilde{T}$ associated to the leaves of the tree $T$.\\
Let suppose that:
$$\exists h>0, \ \forall x \in \Theta, \ \vert 2\eta(x)-1 \vert \geq h \quad {\text{ with }} \quad \eta(x)=\mathbb{P}(Y=1\vert X=x)$$
Then:
\begin{enumerate}
\item[(i)]$\underset{u \in S_T, \ l(s,u) \leq \varepsilon^2}{\sup}
  d(s,u) \leq w(\varepsilon)$ with $w(x)=\frac{1}{\sqrt{h}}x$,
\item[(ii)]$\exists \phi_T : \mathbb{R}^+ \rightarrow \mathbb{R}^+$ such that:
\begin{enumerate}
\item[$\bullet$] $\phi_T(0)=0$,
\item[$\bullet$] $x \rightarrow \frac{\phi_T(x)}{x}$ is non increasing,
\item[$\bullet$] $\forall \sigma \geq w(\sigma_T), \ \ \sqrt{n}
  \mathbb{E}\left[ \underset{u \in S_T, \ d(u,v) \leq \sigma}{\sup}
  \vert \bar{\gamma}_n(u)-\bar{\gamma}_n(v) \vert \right] \leq
  \phi_T(\sigma)$,
\end{enumerate}
with $\sigma_T$ the positive solution of $\phi_T(w(x))=\sqrt{n}x^2$.
\item[(iii)] $\sigma_T^2 \leq \frac{K_3^2 \vert T \vert}{nh}$.
\end{enumerate}
\end{Lemma}

\begin{proof}
The first point $(i)$ is easy to obtain from the following expression of $l$:
$$ l(s,u) = \mathbb{E} \left( \left| s(X)-u(X) \right| \left| 2\eta(X)-1 \right| \right)$$
The existence of the function $\phi_T$ has been proved by Massart and N\'ed\'elec (\cite{Nedelec}). They also give an upper bound of $\sigma_T^2$ based on Sauer's lemma. The upper bound of $\sigma_T^2$ is better than the one of (\cite{Nedelec}) because it has been adapted to the structure of $S_T$.
\end{proof}

Thanks to lemma (\ref{Nedelec}) and (\ref{Maximal}), we deduce the next one.

\begin{Lemma} \label{combi}
Let $(X_1,Y_1),...,(X_n,Y_n)$ a sample taking its values
in some measurable space ${\Theta \times \{0,1\}}$, with common
distribution $P$. Let $T$ a tree, $S_T$ the space associated, $h$ the
margin and $K_3$ the universal constant which appear in the lemme
\ref{Nedelec}.
If $2x \geq \frac{K_3 \sqrt{ \vert T \vert } }{ \sqrt{n}h } $, then:
$$ \mathbb{E} \left[ \underset{u \in S_T}{\sup}
\frac{ \vert \bar{\gamma}_n(u)-\bar{\gamma}_n(v) \vert }{ d^2(u,v)+(2x)^2 }
\right]
\leq \frac{2K_3 \sqrt{ \vert T \vert } }{x \sqrt{n} }$$
\end{Lemma}

\section{Proofs} \label{Proofs}

\subsection{Classification}

\subsubsection{Proof of the proposition \ref{propc1}: }

\noindent Let M $\in \mathcal{P}(\Lambda)$, $ T \preceq T_{max}^{(M)}$
and $s_{M,T} \in S_{M,T}$.
\noindent We let
\begin{align}
  & \bullet w_{M^{\prime},T^{\prime}}(u)  =
  (d(s,s_{M,T})+d(s,u))^{2}+y_{M^{\prime},T^{\prime}}^{2} \nonumber \\
 &  \bullet V_{M^{\prime},T^{\prime}}  =  \underset{u \in
  S_{M^{\prime},T^{\prime}}}{\sup} \frac { \vert
  \bar{\gamma_{n_2}}(u)-\bar{\gamma_{n_2}}(s_{M,T})
  \vert}{w_{M^{\prime},T^{\prime}}(u)} \nonumber
 \end{align}
\noindent where $y_{M^{\prime},T^{\prime}}$ is a parameter that will be chosen later.

\noindent Following the proof of theorem 4.2 in (\cite{Toulouse}), we get
\begin{eqnarray}
 l(s,\tilde{s})  \leq   l(s,s_{M,T}) +  w_{\widehat{M,T}}(\tilde{s})
 \times V_{\widehat{M,T}}+pen(M,T) -  pen( \widehat{M,T}) \label{L}
\end{eqnarray}
\noindent To control $ V_{\widehat{M,T}}$, we check a uniform overestimation of
$V_{M^{\prime},T^{\prime}}$. To do this, we apply the Talagrand's
concentration inequality, written in lemma \ref{Talagrand}, to
$V_{M^{\prime},T^{\prime}}$. So we obtain that for any
$(M^{\prime},T^{\prime})$, and for any $x>0$
$$
\mathbb{P} \left ( V_{M^{\prime},T^{\prime}} \geq K_1 \mathbb{E}
\left[ V_{M^{\prime},T^{\prime}} \right]+K_2 \left(
\sqrt{\frac{x}{2n_2}}y_{M^{\prime},T^{\prime}}
^{-1}+\frac{x}{n_2}y_{M^{\prime},T^{\prime}} ^{-2}\ \right)  \right)
\leq e^{-x}
$$
\noindent where $K_1$ and $K_2$ are universal positive constants.

\noindent Setting $x=x_{M^{\prime},T^{\prime}}+\xi $, with $\xi >0$
and the weights {$x_{M^{\prime},T^{\prime}}=a \vert T^{\prime} \vert +b
\vert M^{\prime} \vert \left ( 1+ \log \left ( \frac{p}{\vert
  M^{\prime \vert}} \right) \right)$}, as defined in lemma \ref{lemme1-poids}, and
summing all those inequalities with respect to
$(M^{\prime},T^{\prime})$, we derive a set $\Omega_{\xi,(M,T)}$ such
that:
\begin{itemize}
 \item [$\bullet$]  $\mathbb{P} \bigg({\Omega}_{\xi,(M,T)}^c \vert
   \mathcal{L}_1 {\text{ and }} \{X_i,\ (X_i,Y_i) \in \mathcal{L}_2\}\bigg)
\leq e^{-\xi} \Sigma(a,b)$
 \item[$\bullet$] on $\Omega_{\xi,(M,T)}$, $\forall
 (M^{\prime},T^{\prime})$,
\begin{equation}
\hspace*{1cm} V_{M^{\prime},T^{\prime}} \leq K_{1}
\mathbb{E}\left[V_{M^{\prime},T^{\prime}} \right] +
K_{2}\left(\sqrt{\frac{x_{M^{\prime},T^{\prime}}+\xi}{2n_2}}{y_{M^{\prime},T^{\prime}}}^{-1}+\frac{x_{M^{\prime},T^{\prime}}+\xi}{n_2}{y_{M^{\prime},T^{\prime}}}^{-2}\right)
\label{V}
\end{equation}
\item[]
\end{itemize}
\noindent  Now we overestimate $\mathbb{E}\left[ V_{M^{\prime},T^{\prime}}
\right] $.\\
\noindent Let $u_{M^{\prime},T^{\prime}} \in
S_{M^{\prime},T^{\prime}}$  such that
$d(s,u_{M^{\prime},T{\prime}}) \leq \underset{u \in
  S_{M^{\prime},T^{\prime}}}{\inf} d(s,u)$.\\
\noindent Then
$$
\mathbb{E}\left[ V_{M^{\prime},T^{\prime}} \right] \leq
\mathbb{E}\left[ \frac{\vert
    \bar{\gamma_{n_2}}(u_{M^{\prime},T^{\prime}})-\bar{\gamma_{n_2}}(s_{M,T}) \vert }{\underset{u \in S_{M^{\prime},T^{\prime}}}{\inf} (w_{M^{\prime},T^{\prime}}(u))}\right] + \mathbb{E}\left[ \underset{u \in S_{M^{\prime},T^{\prime}}}{\sup}\left(  \frac {\vert \bar{\gamma_{n_2}}(u)-\bar{\gamma_{n_2}}(u_{M^{\prime},T^{\prime}}) \vert}{w_{M^{\prime},T^{\prime}}(u)}\right)\right]
$$
\noindent We prove:
$$
\mathbb{E}\left[ \frac{\vert
    \bar{\gamma_{n_2}}(u_{M^{\prime},T^{\prime}})-\bar{\gamma_{n_2}}(s_{M,T}) \vert } {\underset{u \in S_{M^{\prime},T^{\prime}}}{\inf} (w_{M^{\prime},T^{\prime}}(u))}\right] \leq \frac{1}{\sqrt{n_2}y_{M^{\prime},T^{\prime}}}
$$
\noindent For the second term, we have
$$
\mathbb{E}\left[ \underset{u \in
    S_{M^{\prime},T^{\prime}}}{\sup}\left(\frac {\vert
    \bar{\gamma_{n_2}}(u)-\bar{\gamma_{n_2}}(u_{M^{\prime},T^{\prime}}) \vert}{w_{M^{\prime},T^{\prime}}(u)}\right)\right] \leq 4 \mathbb{E}\left[ \underset{u \in
    S_{M^{\prime},T^{\prime}}}{\sup}\left(\frac {\vert
    \bar{\gamma_{n_2}}(u)-\bar{\gamma_{n_2}}(u_{M^{\prime},T^{\prime}}) \vert}{d^2(u,u_{M^{\prime},T^{\prime}})+(2y_{M^{\prime},T^{\prime}})^2}\right)\right]
$$

\noindent By application of the lemma {\ref{combi}} for $2y_{M^{\prime},T^{\prime}} \geq \frac{K_3 \sqrt{\vert T^{\prime}\vert }}{\sqrt{n_2}h}$, we deduce
$$
\mathbb{E}\left[ \underset{u \in
    S_{M^{\prime},T^{\prime}}}{\sup}\left(\frac {\vert
    \bar{\gamma_{n_2}}(u)-\bar{\gamma_{n_2}}(u_{M^{\prime},T^{\prime}}) \vert}{w_{M^{\prime},T^{\prime}}(u)}\right)\right] \leq \frac{8K_3 \sqrt{\vert T^{\prime} \vert } }{\sqrt{n_2}y_{M^{\prime},T^{\prime}}}
$$
\noindent Thus from (\ref{V}), we know that on ${\Omega}_{\xi,(M,T)}$ and $\forall (M^{\prime},T^{\prime})$
$$
 V_{M^{\prime},T^{\prime}} \leq
 \frac{K_{1}}{\sqrt{n_2}y_{M^{\prime},T^{\prime}}}\left(8K_3
 \sqrt{\vert T^{\prime} \vert}
+1\right)+
K_{2}\left(\sqrt{\frac{x_{M^{\prime},T^{\prime}}+\xi
 }{2n_2}}{y_{M^{\prime},T^{\prime}}}^{-1}+
 \frac{x_{M^{\prime},T^{\prime}}+\xi}{n_2}{y_{M^{\prime},T^{\prime}}}^{-2}\right)
$$
\noindent For $y_{M^{\prime},T^{\prime}}=3K \left( \frac{K_{1}}{\sqrt{n_2}}\left(8K_3
\sqrt{\vert T^{\prime}
  \vert}+1\right)+K_{2}\sqrt{\frac{x_{M^{\prime},T^{\prime}}+\xi
  }{2n_2}}+ \frac{1}{\sqrt{3K}}\sqrt{K_2
  \frac{x_{M^{\prime},T^{\prime}}+\xi}{n_2}}\right)$ \\ with $K \geq \frac{1}{48K_1h}$, we get:
$$
V_{M^{\prime},T^{\prime}}  \leq  \frac {1}{K}
$$

\noindent By overestimating $w_{\widehat{M,T}}(\tilde{s})$,
$y_{\widehat{M,T}}^2$ and replacing  all of those results in
(\ref{L}), we get
\begin{eqnarray*}
\left (1-\frac{2}{Kh}\right)l\left(s,\tilde{s}\right) & \leq &
\left(1+\frac{2}{Kh}\right)
l\left(s,s_{M,T}\right)-pen(\widehat{M,T})+pen(M,T) \\
& & +18K\left(\frac{64{K_{1}}^{2}{K_{3}}^{2}}{n_2}\vert \hat{T} \vert
+2K_2\frac{x_{\widehat{M,T}}}{n_2}\left(\sqrt{\frac{K_2}{2}}+\frac{1}{\sqrt{3K}}\right)^2\right)\\
& &
+18K\left(\frac{2K_1^2}{n_2}+2K_2\frac{\xi}{n_2}\left(\sqrt{\frac{K_2}{2}}+\frac{1}{\sqrt{3K}}\right)^2\right)
\end{eqnarray*}

\noindent We let $ K=\frac{2}{h} \frac{C_1+1}{C_1-1}$ with $C_1 >1$.\\
\noindent Taking a penalty $pen(\widehat{M,T})$ which balances all the
terms in $\left(\widehat{M,T}\right)$, i.e.
$$
pen(M,T) \geq
\frac{36(C_1+1)}{h(C_1-1)}\left(\frac{64K_1^2K_3^2}{n_2}\vert T \vert
+
2K_2\frac{x_{M,T}}{n_2}\left(\sqrt{\frac{K_2}{2}}+\sqrt{\frac{C_1-1}{6(C_1+1)}}\right)^2\right)
$$

\noindent We obtain that on $\Omega_{\xi,(M,T)}$
$$
l(s,\tilde{s}) \leq C_1 \bigg( l(s,s_{M,T})+pen(M,T) \bigg) + \frac{C}{n_2h}\xi
$$
\noindent Integrating with respect to $\xi$ and by minimizing , we get
$$
\mathbb{E} \bigg[ l(s,\tilde{s}) \vert \mathcal{L}_1 \bigg]
\leq C_1\ \underset{M,T}{\inf} \bigg\{ l(s,S_{M,T}) + pen(M,T) \bigg\}
+ \frac{C}{n_2h} \Sigma(a,b) $$

\noindent In brief, with a penalty function such that\\
\noindent $
\forall M \in \mathcal{P}(\Lambda), \ \ \forall T \preceq T_{max}^{(M)}
$
\begin{eqnarray*}
pen(M,T) & = & \alpha \frac{\vert T \vert}{n_2h}+ \beta \frac{\vert M \vert}{n_2h}\left(1+log\left(\frac{p}{\vert M \vert} \right) \right) \\
& \geq & \frac{36(C_1+1)}{C_1-1}\left(64K_1^2K_3^2+2aK_2\left( \sqrt{\frac{K_2}{2}}+\sqrt{\frac{h(C_1-1)}{6(C_1+1)}}\right)^2\right)\frac{\vert T \vert}{n_2h}\\
& & +\frac{36(C_1+1)}{C_1-1} 2K_2 \left(\sqrt{\frac{K_2}{2}}+\sqrt{\frac{h(C_1-1)}{6(C_1+1)}}\right)^2b \frac{\vert M \vert}{n_2h}\left(1+log\left(\frac{p}{\vert M \vert}\right)\right)
\end{eqnarray*}
\noindent we have:
$$
\mathbb{E} \bigg[ l(s,\tilde{s}) \vert \mathcal{L}_1 \bigg]
\leq C_1\ \underset{M,T}{\inf} \bigg\{ l(s,S_{M,T}) + pen(M,T) \bigg\}
+ \frac{C}{n_2h} \Sigma(a,b) $$

\noindent We notice that, the two constants $\alpha_0$ and $\beta_0$,
which appear in the proposition \ref{propc1}, are defined by
$$\alpha_0 = 36\left(64K_1^2K_3^2+4 \log 2 K_2 \left(
\sqrt{\frac{K_2}{2}}+\frac{1}{\sqrt{6}} \right)^2\right)  \ \ {\text{and}} \ \ \beta_0 = 72 K_2 \left(\sqrt{\frac{K_2}{2}}+\frac{1}{\sqrt{6}} \right)^2 $$
\hfill $\Box$

\subsubsection{Proof of the proposition \ref{propc2}: }

\noindent For $M$, $M^{\prime}$  $\in \mathcal{P}(\Lambda)$, $ T \preceq T_{max}^{(M)}$, $ T^{\prime} \preceq T_{max}^{(M^{\prime})}$
and $s_{M,T} \in S_{M,T}$.
We let
\begin{align}
  & \bullet w_{(M^{\prime},T^{\prime}),(M,T)}(u)  =
  (d(s,s_{M,T})+d(s,u))^{2}+(y_{M^{\prime},T^{\prime}}+y_{M,T})^{2} \nonumber \\
 &  \bullet V_{(M^{\prime},T^{\prime}),(M,T)}  =  \underset{u \in
  S_{M^{\prime},T^{\prime}}}{\sup} \frac { \vert
  \bar{\gamma_{n_2}}(u)-\bar{\gamma_{n_2}}(s_{M,T})
  \vert}{w_{(M^{\prime},T^{\prime}),(M,T)}(u)} \nonumber
 \end{align}
\noindent where $y_{M^{\prime},T^{\prime}}$ and $y_{M,T}$ are parameters that will be chosen later.

\noindent Following the proof of theorem 4.2 in (\cite{Toulouse}), we get
\begin{eqnarray}
 l(s,\tilde{s})  \leq   l(s,s_{M,T}) +  w_{(\widehat{M,T}),(M,T)}(\tilde{s})
 \times V_{(\widehat{M,T}),(M,T)}+pen(M,T) -  pen( \widehat{M,T}) \label{L1}
\end{eqnarray}
\noindent To control $ V_{(\widehat{M,T}),(M,T)}$, we check a uniform overestimation of
$V_{(M^{\prime},T^{\prime}),(M,T)}$. To do this, we apply the Talagrand's
concentration inequality, written in lemma \ref{Talagrand}, to
$V_{(M^{\prime},T^{\prime}),(M,T)}$, for $(M^{\prime},T^{\prime}) \in \mathcal{P}(\Lambda) \times \mathcal{M}_{n_1,M^{\prime}}$ and $(M,T) \in \mathcal{P}(\Lambda) \times \mathcal{M}_{n_1,M}$ . So we obtain that for any
$(M^{\prime},M) \in \mathcal{P}(\Lambda)^2$, any $T^{\prime} \in  \mathcal{M}_{n_1,M^{\prime}}$, any $M \in \mathcal{M}_{n_1,M}$ and any $x>0$,
$$
\mathbb{P} \left ( V_{(M^{\prime},T^{\prime}),(M,T)} \geq K_1 \mathbb{E}
\left[ V_{(M^{\prime},T^{\prime}),(M,T)} \right]+K_2 \left(
\sqrt{\frac{x}{2n_2}}y_{M^{\prime},T^{\prime}}
^{-1}+\frac{x}{n_2}y_{M^{\prime},T^{\prime}} ^{-2}\ \right)  \right)
\leq e^{-x}
$$
\noindent where $K_1$ and $K_2$ are universal positive constants.

\noindent Setting $x=x_{M^{\prime},T^{\prime}}+x_{M,T}+\xi $, with $\xi >0$
and the weights {$x_{M^{\prime},T^{\prime}}=\left(a +(\vert M^{\prime} \vert +1)\left(1+log\left(\frac{n_1}{\vert M^{\prime} \vert +1}\right)\right)\right)\vert T^{\prime} \vert +b
\vert M^{\prime} \vert \left ( 1+ \log \left ( \frac{p}{\vert
  M^{\prime} \vert} \right) \right)$}, as defined in lemma \ref{lemme2-poids}, and
summing all those inequalities with respect to
$(M^{\prime},T^{\prime})$ and $(M,T)$, we derive a set $\Omega_{\xi}$ such
that:
\begin{itemize}
 \item [$\bullet$]  $\mathbb{P} \bigg({\Omega}_{\xi}^c \vert
   \{X_i,\ (X_i,Y_i) \in \mathcal{L}_1\}\bigg)
\leq e^{-\xi} (\Sigma(a,b))^2$
 \item[$\bullet$] on $\Omega_{\xi}$, $\forall
 (M^{\prime},T^{\prime})$, $(M,T)$
\begin{eqnarray}
\hspace*{1cm} V_{(M^{\prime},T^{\prime}),(M,T)} & \leq & K_{1}
\mathbb{E}\left[V_{(M^{\prime},T^{\prime}),(M,T)} \right] +
K_{2}\left(\sqrt{\frac{x_{M^{\prime},T^{\prime}}+x_{M,T}+\xi}{2n_1}}(y_{M^{\prime},T^{\prime}}+y_{M,T})^{-1} \right) \nonumber \\
& & +K_{2}\left( \frac{x_{M^{\prime},T^{\prime}}+x_{M,T}+\xi}{n_1}(y_{M^{\prime},T^{\prime}}+y_{M,T})^{-2}\right)
\label{V1}
\end{eqnarray}
\item[]
\end{itemize}

\noindent Now we overestimate $\mathbb{E}\left[ V_{(M^{\prime},T^{\prime}),(M,T)}
\right] $.\\
\noindent Let $u_{M^{\prime},T^{\prime}} \in
S_{M^{\prime},T^{\prime}}$  such that
$d(s,u_{M^{\prime},T{\prime}}) \leq \underset{u \in
  S_{M^{\prime},T^{\prime}}}{\inf} d(s,u)$.\\
\noindent Then
$$
\mathbb{E}\left[ V_{(M^{\prime},T^{\prime}),(M,T)} \right] \leq
\mathbb{E}\left[ \frac{\vert
    \bar{\gamma_{n_1}}(u_{M^{\prime},T^{\prime}})-\bar{\gamma_{n_1}}(s_{M,T}) \vert }{\underset{u \in S_{M^{\prime},T^{\prime}}}{\inf} (w_{(M^{\prime},T^{\prime}),(M,T)}(u))}\right] + \mathbb{E}\left[ \underset{u \in S_{M^{\prime},T^{\prime}}}{\sup}\left(  \frac {\vert \bar{\gamma_{n_1}}(u)-\bar{\gamma_{n_1}}(u_{M^{\prime},T^{\prime}}) \vert}{w_{(M^{\prime},T^{\prime}),(M,T)}(u)}\right)\right]
$$
\noindent We prove:
$$
\mathbb{E}\left[ \frac{\vert
    \bar{\gamma_{n_1}}(u_{M^{\prime},T^{\prime}})-\bar{\gamma_{n_1}}(s_{M,T}) \vert } {\underset{u \in S_{M^{\prime},T^{\prime}}}{\inf} (w_{(M^{\prime},T^{\prime}),(M,T)}(u))}\right] \leq \frac{1}{\sqrt{n_1}(y_{M^{\prime},T^{\prime}}+y_{M,T})}
$$

\noindent For the second term, we have
$$
\mathbb{E}\left[ \underset{u \in S_{M^{\prime},T^{\prime}}}{\sup}\left(  \frac {\vert \bar{\gamma_{n_1}}(u)-\bar{\gamma_{n_1}}(u_{M^{\prime},T^{\prime}}) \vert}{w_{(M^{\prime},T^{\prime}),(M,T)}(u)}\right)\right] \leq 4\mathbb{E}\left[ \underset{u \in S_{M^{\prime},T^{\prime}}}{\sup}\left(  \frac {\vert \bar{\gamma_{n_1}}(u)-\bar{\gamma_{n_1}}(u_{M^{\prime},T^{\prime}}) \vert}{d^2(u,u_{M^{\prime},T^{\prime}})+(2(y_{M^{\prime},T^{\prime}}+y_{M,T}))^2}\right)\right]
$$

\noindent By application of lemma \ref{combi} for $2y_{M^{\prime},T^{\prime}} \geq \frac{K_3 \sqrt{\vert T^{\prime}\vert }}{\sqrt{n_1}h}$,
$$
\mathbb{E}\left[ \underset{u \in
    S_{M^{\prime},T^{\prime}}}{\sup}\left(\frac {\vert
    \bar{\gamma_{n_1}}(u)-\bar{\gamma_{n_1}}(u_{M^{\prime},T^{\prime}}) \vert}{w_{(M^{\prime},T^{\prime}),(M,T)}(u)}\right)\right] \leq \frac{8K_3 \sqrt{\vert T^{\prime} \vert } }{\sqrt{n_1}(y_{M^{\prime},T^{\prime}}+y_{M,T})}
$$

\noindent  Thus from (\ref{V1}), we know that on ${\Omega}_{\xi}$ and $\forall (M^{\prime},T^{\prime}), \ (M,T)$
\begin{eqnarray*}
 V_{(M^{\prime},T^{\prime}),(M,T)} & \leq &
 \frac{K_{1}}{\sqrt{n_1}(y_{M^{\prime},T^{\prime}}+y_{M,T})}\left(8K_3
 \sqrt{\vert T^{\prime} \vert}
+1\right)+
K_{2}\left(\sqrt{\frac{x_{M^{\prime},T^{\prime}}+x_{M,T}+\xi
 }{2n_1}}{(y_{M^{\prime},T^{\prime}}+y_{M,T})}^{-1}+ \right) \\
 & &
K_{2}\left( \frac{x_{M^{\prime},T^{\prime}}+x_{M,T}+\xi}{n_1}{(y_{M^{\prime},T^{\prime}}+y_{M,T})}^{-2}\right)
\end{eqnarray*}
\noindent  For $y_{M^{\prime},T^{\prime}}=3K \left( \frac{K_{1}}{\sqrt{n_1}}\left(8K_3 \sqrt{\vert T^{\prime} \vert}+1\right)+K_{2}\sqrt{\frac{x_{M^{\prime},T^{\prime}}+\xi/2}{2n_1}}+ \frac{1}{\sqrt{3K}}\sqrt{K_2 \frac{x_{M^{\prime},T^{\prime}}+\xi/2}{n_1}}\right)$ \\ with $K \geq \frac{1}{48K_1h}$, we get:
$$
V_{(M^{\prime},T^{\prime}),(M,T)}  \leq  \frac {1}{K}
$$

\noindent By overestimating $w_{(\widehat{M,T}),(M,T)}(\tilde{s})$,
$y_{\widehat{M,T}}^2$ and replacing  all of those results in
(\ref{L1}), we get

\begin{eqnarray*}
\left( 1-\frac{2}{K h}\right) l \left( s,\tilde{s}\right) & \leq &
\left( 1+\frac{2}{K h}\right) l\left(s,s_{M,T}\right)-pen(\widehat{M,T})+pen(M,T) \\
& & +36K \left(\frac{64K_1^2K_3^2}{n_1}\vert \hat{T} \vert + \frac{64K_1^2K_3^2}{n_1}\vert T \vert +2K_2\frac{x_{\widehat{M,T}}}{n_1}\left(\sqrt{\frac{K_2}{2}}+\frac{1}{\sqrt{3K}}\right)^2 \right) \\
& & +36K\left( 2K_2\frac{x_{M,T}}{n_1}\left(\sqrt{\frac{K_2}{2}}+\frac{1}{\sqrt{3K}}\right)^2\right)+36K\left(\frac{4K_1^2}{n_1}+2K_2\frac{\xi}{n_1} \left(\sqrt{\frac{K_2}{2}}+\frac{1}{\sqrt{3K}}\right)^2\right)\\
& &
\end{eqnarray*}

\noindent We let $ K=\frac{2}{h} \frac{C_1+1}{C_1-1}$ with $C_1 >1$.\\
\noindent Taking a penalty $pen(\widehat{M,T})$ which balances all the
terms in $\left(\widehat{M,T}\right)$, i.e.
$$
pen(M,T) \geq
\frac{72(C_1+1)}{h(C_1-1)}\left(\frac{64K_1^2K_3^2}{n_1}\vert T \vert
+
2K_2\frac{x_{M,T}}{n_1}\left(\sqrt{\frac{K_2}{2}}+\sqrt{\frac{C_1-1}{6(C_1+1)}}\right)^2\right)
$$

\noindent We obtain that on $\Omega_{\xi}$
$$
l(s,\tilde{s}) \leq 2C_1 \bigg( l(s,s_{M,T})+pen(M,T) \bigg) + \frac{C_2}{n_1h}+\frac{C_3}{n_1h}\xi
$$

\noindent In brief, with a penalty function such that
\noindent $\forall M \in \mathcal{P}(\Lambda), \ \ forall T \preceq T_{max}^{(M)}$
\begin{eqnarray*}
pen(M,T) & = & \alpha \frac{\vert T \vert}{n_1h}\left( 1 + (\vert M \vert +1)\left(1+log \left( \frac{n_1}{\vert M \vert+1} \right) \right) \right) + \beta \frac{\vert M \vert}{n_1h}\left( 1 + log \left( \frac{p}{\vert M \vert} \right) \right) \\
& \geq & \frac{72(C_1+1)}{C_1-1}\left(64K_1^2K_3^2+2K_2\left(\sqrt{\frac{K_2}{2}}+\sqrt{\frac{C_1-1}{6(C_1+1)}} \right)^2\right) \frac{\vert T \vert}{n_1h} \\
& & \times \left(a+(\vert M \vert +1)\left(1 + log \left( \frac{n_1}{\vert M \vert +1} \right)\right)\right)\\
& & + \frac{72(C_1+1)}{C_1-1}2K_2\left(\sqrt{\frac{K_2}{2}}+\sqrt{\frac{C_1-1}{6(C_1+1)}}  \right)^2 b \frac{\vert M \vert}{n_1h} \left( 1+ log \left( \frac{p}{\vert M \vert}\right) \right)
\end{eqnarray*}

\noindent we have
$$
l(s,\tilde{s}) \leq 2C_1 \Big\{ l(s,s_{M,T})+pen(M,T) \Big\} + \frac{C_2}{n_1h}\left(1 + \xi\right)
$$
\noindent We notice that the two constants $\alpha_0$ and $\beta_0$ which appear in the proposition \ref{propc2} are defined by
$$
\alpha_0=72 \left( 64K_1^2K_3^2+2K_2\left( \sqrt{\frac{K_2}{2}}+\frac{1}{\sqrt{6}} \right)^2 \right) \  {\text{ and }} \ \ \beta_0=72 \times 2 \times K_2\left( \sqrt{\frac{K_2}{2}}+\frac{1}{\sqrt{6}} \right)^2
$$

\hfill $\Box$\\

\subsubsection{Proof of the proposition \ref{propcf}: }

 This result is obtained by a direct application of the lemma
\ref{Boucheron} which appears in the subsection \ref{Appendix}.
\hfill $\Box$

\subsection{Regression}

\subsubsection{Proof of the proposition \ref{propM1}: }

\noindent Let $a>2 \log 2$, $b>1$, $\theta \in (0,1)$ and $K > 2-\theta$ four constants.

\noindent Let us denote
$$ s_{M,T} = \underset{u \in S_{M,T}}{argmin}\  \| s -u \|_{n_2}^2  \ \ {\text{and}} \ \
\varepsilon_{M,T} = \underset{u \in S_{M,T}}{argmin} \  \| \varepsilon-u \|_{n_2}^2 $$

\noindent Following the proof of theorem 1 in (\cite{Birge}), we get
\begin{eqnarray} \label{base}
(1-\theta)\| s -\tilde{s} \|_{n_2}^2 = \Delta_{\widehat{M,T}} + \underset{(M,T)}{\inf} R_{M,T}
\end{eqnarray}
\noindent where
\begin{eqnarray*}
\Delta_{M,T} &=& (2-\theta) \| \varepsilon_{M,T} \|_{n_2}^2
- 2 < \varepsilon, s-s_{M,T} >_{n_2}
- \theta \| s -s_{M,T} \|_{n_2}^2  - pen(M,T) \\
R_{M,T} &=& \| s -s_{M,T} \|_{n_2}^2 - \| \varepsilon_{M,T} \|_{n_2}^2
+ 2 < \varepsilon, s-s_{M,T} >_{n_2} + pen(M,T)
\end{eqnarray*}

\noindent We are going first to control $\Delta_{\widehat{M,T}}$
by using concentration inequalities of $ \| \varepsilon_{M,T} \|_{n_2}^2$
and ${- < \varepsilon, s-s_{M,T} >_{n_2}}$. \\

\noindent For any $M$, we denote $$\Omega_{M} = \left\{ \forall t \in
\widetilde{T_{max}^{(M)}} \ \left| \sum_{X_i \in t} \varepsilon_i \right |
\leq \frac{\sigma^2}{\rho} |X_i \in t| \right\} $$

\noindent Thanks to lemma \ref{faux-chi2}, we get that for any $(M,T)$ and any $x > 0$
\begin{eqnarray} \label{I1}
\nonumber \mathbb{P} \Bigg(  \| \varepsilon_{M,T} \|_{n_2}^2 \1_{\Omega_{M}}
& \geq & \frac{\sigma^2}{n_2} |T| + 8 \frac{\sigma^2}{n_2}\sqrt{2|T|x}
+ 4 \frac{\sigma^2}{n_2} x
\ \Big| \mathcal{L}_1 \text{ and } \{ X_i; \ (X_i,Y_i) \in \mathcal{L}_2 \} \Bigg) \\
\leq  e^{-x} & &
\end{eqnarray}
\noindent and
\begin{eqnarray*}
\mathbb{P} \left( \Omega_{M}^c \
\Big| \mathcal{L}_1 \text{ and } \{ X_i; \ (X_i,Y_i) \in \mathcal{L}_2 \} \right)
\leq 2 \frac{n_2}{N_{min}} \exp \left( \frac{- \sigma^2 N_{min} }{4 \rho^2} \right)
\end{eqnarray*}

\noindent Denoting $\Omega = \underset{M}{\bigcap} \Omega_{M}$,
we have
\begin{eqnarray*}
\mathbb{P} \left( \Omega^c \
\Big| \mathcal{L}_1 \text{ and } \{ X_i; \ (X_i,Y_i) \in \mathcal{L}_2 \} \right)
\leq 2^{p+1} \frac{n_2}{N_{min}} \exp \left( \frac{-\sigma^2 N_{min} }{4 \rho^2} \right)
\end{eqnarray*}

\noindent To control ${- < \varepsilon, s-s_{M,T} >_{n_2}}$, we calculate its Laplace transform.
Thanks to assumption (\ref{A}) and $ \|s\|_{\infty} \leq R$,
we have for any $(M,T)$ and any $\lambda \in \left( 0; \ \frac{n_2}{2 \rho R} \right)$,
$$
log \mathbb{E}\left[ e^{-\lambda < \varepsilon,s-s_{M,T} >_{n_2}} \ \Big| \mathcal{L}_1 \text{ and } \{ X_i; \ (X_i,Y_i) \in \mathcal{L}_2 \} \right] \leq  \frac{\lambda^2\sigma^2 \| s-s_{M,T} \|_{n_2}^2}{2n_2\left(1-\lambda\frac{2\rho R}{n_2}  \right)}
$$

\noindent Thus, for any $(M,T)$ and any $x>0$,
\begin{eqnarray} \label{I2}
\nonumber \mathbb{P}\Big( -<\varepsilon, s-s_{M,T} >_{n_2} & \geq &
\frac{\sigma}{\sqrt{n_2}} \| s-s_{M,T} \|_{n_2} \sqrt{2x}
+ \frac{2\rho R }{n_2} x
\ \Big| \mathcal{L}_1 \text{ and } \{ X_i; \ (X_i,Y_i) \in \mathcal{L}_2 \}
\Big) \\
\leq e^{-x} &&
\end{eqnarray}

\noindent Setting $x = x_{M,T} + \xi$ with $\xi >0$ and the weights ${ x_{M,T} = a |T|+ b |M| \left(1 + \log\left(\frac{p}{|M|}\right)\right) }$ as defined in lemma \ref{lemme1-poids}, and summing all inequalities (\ref{I1}) and (\ref{I2}) with respect to $(M,T)$, we derive a set $ E_{\xi} $ such that
\begin{itemize}
\item $ \mathbb{P} \left( E_{\xi}^c \ | \mathcal{L}_1 \text{ and } \{ X_i; \ (X_i,Y_i) \in \mathcal{L}_2 \} \right) \leq  2 e^{-\xi} \Sigma(a,b) $ \\
\item on the set $ E_{\xi} \bigcap \Omega $, for any $(M,T)$,
\begin{eqnarray*}
\Delta_{M,T} & \leq & (2-\theta)\frac{\sigma^2}{n_2}|T| + 8(2-\theta)\frac{\sigma^2}{n_2}\sqrt{2|T|(x_{M,T}+ \xi)}
+ 4(2-\theta) \frac{\sigma^2}{n_2} (x_{M,T}+ \xi) \\
& & + 2 \frac{\sigma}{\sqrt{n_2}} \| s-s_{M,T} \|_{n_2} \sqrt{2(x_{M,T}+ \xi)}  + 4 \frac{\rho R}{n_2} (x_{M,T}+ \xi) \\
& & - \theta \| s-s_{M,T} \|_{n_2}^2 - pen(M,T)
\end{eqnarray*}
\end{itemize}
\noindent where $\Sigma(a,b) = -\log\left( 1-e^{-(a-2 \log 2)} \right) \frac{e^{-(b-1)}}{1-e^{-(b-1)}} $.\\

\noindent Using the two following inequalities
\begin{enumerate}
\item[]
$2 \frac{\sigma}{\sqrt{n_2}} \| s-s_{M,T} \|_{n_2} \sqrt{2(x_{M,T}+ \xi)}
\leq \theta \| s-s_{M,T} \|_{n_2}^2 + \frac{2}{\theta} \frac{\sigma^2}{n_2} (x_{M,T}+ \xi)$,
\item[]
$2 \sqrt{|T|(x_{M,T} + \xi)} \leq \eta |T| + \eta^{-1} (x_{M,T}+ \xi)$
\end{enumerate}
with $\eta = \frac{K+\theta-2}{2-\theta} \frac{1}{4\sqrt{2}} >0$, we derive that on the set $ E_{\xi} \bigcap \Omega $, for any $(M,T)$,

\begin{eqnarray*}
\Delta_{M,T} & \leq & (2-\theta)\frac{\sigma^2}{n_2}\vert T \vert + 8\sqrt{2}(2-\theta)\frac{\sigma^2}{n_2}\sqrt{\vert T \vert(x_{M,T}+\xi)} +\left( 4(2-\theta)+\frac{2}{\theta}+4\frac{\rho}{\sigma^2}R \right)\frac{\sigma^2}{n_2}(x_{M,T}+\xi) -pen(M,T)  \\
& \leq & K \frac{\sigma^2}{n_2} |T| + \left( 4(2-\theta) \left(
1 + \frac{8(2-\theta)}{K+\theta-2} \right)
+ \frac{2}{\theta} + 4 \frac{\rho}{\sigma^2} R  \right)
\frac{\sigma^2}{n_2} (x_{M,T}+ \xi)
- pen(M,T)
\end{eqnarray*}

\noindent Taking a penalty $pen(M,T)$ which compensates for all the other terms in $(M,T)$, i.e.
\begin{eqnarray} \label{pen}
pen(M,T) \geq  K \frac{\sigma^2}{n_2}|T| +
\left[ 4(2-\theta) \left( 1 + \frac{8(2-\theta)}{K+\theta-2} \right)
+ \frac{2}{\theta} + 4 \frac{\rho}{\sigma^2} R  \right]
\frac{\sigma^2}{n_2}x_{M,T} \nonumber \\
&
\end{eqnarray}

\noindent we get that, on the set $ E_{\xi}$
$$\Delta_{\widehat{M,T}} \1_{\Omega}
\leq  \left( 4(2-\theta) \left(
1 + \frac{8(2-\theta)}{K+\theta-2} \right)
+ \frac{2}{\theta} + 4 \frac{\rho}{\sigma^2} R  \right)
 \frac{\sigma^2}{n_2} \xi$$

\noindent Integrating with respect to $\xi$, we derive
\begin{eqnarray} \label{delta}
& \mathbb{E}\left[ \Delta_{\widehat{M,T}} \1_{\Omega}
\Big| \mathcal{L}_1 \right] \leq
2 \left( 4(2-\theta) \left(
1 + \frac{8(2-\theta)}{K+\theta-2} \right)
+ \frac{2}{\theta} + 4 \frac{\rho}{\sigma^2} R  \right)
\frac{\sigma^2}{n_2} \Sigma(a,b) \\
& \nonumber
\end{eqnarray}

\noindent We are going now to control $\mathbb{E}\left[ \underset{(M,T)}{\inf} R_{M,T}
\1_{\Omega} \Big| \mathcal{L}_1 \right]$.

\noindent In the same way we deduced (\ref{I2}) from assumption (\ref{A}), we get that for any $(M,T)$ and any $x>0$
\begin{eqnarray*}
\mathbb{P}\Big( <\varepsilon, s-s_{M,T} >_{n_2} & \geq &
\frac{\sigma}{\sqrt{n_2}} \| s-s_{M,T} \|_{n_2} \sqrt{2x}
+ \frac{2\rho R}{n_2} x
\ \Big| \mathcal{L}_1 \text{ and } \{ X_i; \ (X_i,Y_i) \in \mathcal{L}_2 \}
\Big) \\
\leq e^{-x} &&
\end{eqnarray*}

\noindent Thus we derive a set $F_{\xi}$ such that
\begin{itemize}
\item $ \mathbb{P} \left( F_{\xi}^c \ | \mathcal{L}_1 \text{ and } \{ X_i; \ (X_i,Y_i) \in \mathcal{L}_2 \} \right) \leq  e^{-\xi} \Sigma(a,b) $ \\
\item on the set $ F_{\xi} $, for any $(M,T)$,
\begin{eqnarray*}
<\varepsilon, s-s_{M,T} >_{n_2} & \leq &
\frac{\sigma}{\sqrt{n_2}} \| s-s_{M,T} \|_{n_2} \sqrt{2\left(x_{M,T}+\xi \right)}
+ \frac{2\rho R}{n_2}\left(x_{M,T}+\xi \right)
\end{eqnarray*}
\end{itemize}

\noindent It follows from definition of $R_{M,T}$ that on the set $F_{\xi}$, for any $(M,T)$,
\begin{eqnarray*}
R_{M,T}= & \leq & \| s - s_{M,T} \|_{n_2}^2 + 2\frac{\sigma}{\sqrt{n_2}} \|s - s_{M,T}\|_{n_2}\sqrt{2(x_{M,T}+\xi)}+\frac{4\rho R}{n_2}(x_{M,T}+\xi)+pen(M,T) \\
& \leq &  2 \| s -s_{M,T}\|_{n_2}^2+\left( 2+4\frac{\rho}{\sigma^2}R\right)\frac{\sigma^2}{n_2}(x_{M,T}+\xi)+pen(M,T) \\
& \leq & 2 \| s -s_{M,T}\|_{n_2}^2+ 2pen(M,T) + \left( 2+4\frac{\rho}{\sigma^2}R\right)\frac{\sigma^2}{n_2} \xi\\
& &
\end{eqnarray*}
\noindent And
\begin{eqnarray} \label{reste}
\mathbb{E}\left[ \underset{(M,T)}{\inf} R_{M,T}
\1_{\Omega} \Big| \mathcal{L}_1 \right]
& \leq & 2 \underset{(M,T)}{\inf}
\left\{ \mathbb{E}\left[ \| s-s_{M,T} \|_{n_2}^2 \Big| \mathcal{L}_1 \right]
+ pen(M,T) \right\} \\
\nonumber & & + \left( \frac{2}{\theta} + 4 \frac{\rho}{\sigma^2} R  \right)\frac{\sigma^2}{n_2} \Sigma(a,b)
\end{eqnarray}

\noindent We conclude from (\ref{base}), (\ref{delta}) and (\ref{reste}) that
\begin{eqnarray*}
(1-\theta) \mathbb{E}\left[ \| s -\tilde{s} \|_{n_2}^2 \1_{\Omega}
\Big| \mathcal{L}_1 \right] \leq &
2 \underset{(M,T)}{\inf} \left\{ \mathbb{E}\left[ \| s -s_{M,T} \|_{n_2}^2
\Big| \mathcal{L}_1 \right] + pen(M,T) \right\} \\
& +  \left( 8(2-\theta) \left(
1 + \frac{8(2-\theta)}{K+\theta-2} \right)
+ \frac{6}{\theta} + 12 \frac{\rho}{\sigma^2} R  \right)
\frac{\sigma^2}{n_2} \Sigma(a,b) &
\end{eqnarray*}

\noindent It remains to control $\mathbb{E}\left[ \| s -\tilde{s} \|_{n_2}^2
\1_{\Omega^c}
\Big| \mathcal{L}_1 \right]$.
\begin{eqnarray*}
\mathbb{E}\left[ \| s -\tilde{s} \|_{n_2}^2 \1_{\Omega^c}
\Big| \mathcal{L}_1 \right] & = & \mathbb{E}\left[ \| s -s_{\widehat{M,T}} \|_{n_2}^2 \1_{\Omega^c}
\Big| \mathcal{L}_1 \right] + \mathbb{E}\left[ \| \varepsilon_{\widehat{M,T}} \|_{n_2}^2 \1_{\Omega^c}
\Big| \mathcal{L}_1 \right] \\
& \leq & \mathbb{E}\left[ \| s \|_{n_2}^2 \1_{\Omega^c}
\Big| \mathcal{L}_1 \right] + \sum_{M} \mathbb{E}\left[ \| \varepsilon_{M,T_{max}^{(M)}} \|_{n_2}^2 \1_{\Omega^c}
\Big| \mathcal{L}_1 \right]\\
& \leq &  R^2\mathbb{P}\left( \Omega^c \Big| \mathcal{L}_1 \right)
+ \sum_M \sqrt{ \mathbb{E}\left[ \| \varepsilon_{M,T_{max}^{(M)}}  \|_{n_2}^4
\Big| \mathcal{L}_1 \right] }
\sqrt{ \mathbb{P}\left( \Omega^c \Big| \mathcal{L}_1 \right)}
\end{eqnarray*}

\noindent As
\begin{eqnarray*}
 \mathbb{E}\left[ \| \varepsilon_{M,T_{max}^{(M)}}  \|_{n_2}^4
\Big| \mathcal{L}_1 \right] & \leq & \frac{\sigma^4\vert T_{max}^[M) \vert^2}{n_2^2}+\frac{c^2(\rho,\sigma) \vert T_{max}^{(M)}\vert}{n_2^2N_{min}}+\frac{3\sigma^4 \vert T_{max}^{(M)} \vert}{n_2^2} \\
& \leq &
\frac{C^2(\rho,\sigma) }{N_{min}^2}
\end{eqnarray*}
\noindent where $C(\rho, \sigma)$ is a constant which depends only on $\rho$ and $\sigma$.

\noindent Thus we have
\begin{eqnarray*}
\mathbb{E}\left[ \| s -\tilde{s} \|_{n_2}^2 \1_{\Omega^c}
\Big| \mathcal{L}_1 \right] & \leq &
R^2\mathbb{P}\left( \Omega^c \Big| \mathcal{L}_1 \right)
+ 2^p \frac{C(\rho,\sigma)}{N_{min}}
\sqrt{ \mathbb{P}\left( \Omega^c \Big| \mathcal{L}_1 \right)}
\end{eqnarray*}

\noindent Let us recall that
\begin{eqnarray*}
\mathbb{P} \left( \Omega^c \
\Big| \mathcal{L}_1 \right)
& \leq & 2^{p+1} \frac{n_2}{N_{min}} \exp \left( \frac{-\sigma^2 N_{min} }{4 \rho^2} \right)
\end{eqnarray*}

\noindent For $p \leq \log n_2$ and $N_{min} \geq \frac{24 \rho^2}{ \sigma^2} \log n_2$,
\begin{itemize}
\item $ 2^p \sqrt{ \mathbb{P}\left( \Omega_{\delta}^c \Big| \mathcal{L}_1
\right)}
\leq \frac{\sigma}{\sqrt{12}\rho} \frac{1}{n_2 \sqrt{\log n_2}} $
\item $ \mathbb{P} \left( \Omega_{\delta}^c \ \Big| \mathcal{L}_1 \right)
\leq \frac{\sigma^2}{12\rho^2} \frac{1}{n_2^4 \log n_2}$
\end{itemize}

\noindent It follows that
\begin{eqnarray*}
\mathbb{E}\left[ \| s -\tilde{s} \|_{n_2}^2 \1_{\Omega^c} \Big| \mathcal{L}_1 \right] & \leq &
C^{\prime}(\rho, \sigma, R) \frac{1}{n_2 (\log n_2)^{3/2}}
\end{eqnarray*}

\noindent Finally, we have the following result:

\noindent Denoting by $\Upsilon = \left[ 4(2-\theta) \left( 1 + \frac{8(2-\theta)}{K+\theta-2} \right)
+ \frac{2}{\theta} \right]$
  and
taking a penalty which satisfies
$\forall \ M \in \mathcal{P}(\Lambda)$ $ \forall \ T \preceq  T_{max}^{(M)}$

\begin{eqnarray*}
pen(M,T) & \geq & \left( \left( K + a \Upsilon \right)\sigma^2  + 4 a \rho R \right)
\frac{|T|}{n_2}
+ \left( b \Upsilon \sigma^2 + 4b \rho R \right)
\frac{|M|}{n_2} \left( 1 + \log\left( \frac{p}{|M|} \right) \right)
\end{eqnarray*}

\noindent if $p \leq \log n_2$ and $N_{min} \geq \frac{24 \rho^2}{\sigma^2} \log n_2$, we have,
\begin{eqnarray*}
(1-\theta) \mathbb{E} \left[ \| s -\tilde{s} \|_{n_2}^2 \
|\mathcal{L}_1 \right] & \leq &  2 \underset{(M,T)}{\inf} \left\{  \underset{u \in S_{M,T}}{\inf} \left\| s - u \right\|_{\mu}^2  +  pen(M,T) \right\} \\
& & + \left( 2 \Upsilon +2 + 12\frac{\rho}{\sigma^2}R  \right) \frac{\sigma^2}{n_2} \Sigma(a,b) \\
& & + (1-\theta) C^{\prime}(\rho, \sigma, R) \frac{1}{n_2(\log n_2)^{3/2}}
\end{eqnarray*}

\noindent We deduce the proposition by taking $K=2$,
$\theta \rightarrow 1$, $a \rightarrow 2 \log 2$ and $b \rightarrow 1$. \hfill $\Box$\\

\subsubsection{Proof of the proposition \ref{propM2}: }

\noindent Let $a>0$, $b>1$, $\theta \in (0,1)$ and $K>2-\theta$ four constants.

\noindent To follow the preceding proof, we have to consider the ``deterministic'' bigger collection of models:
\begin{eqnarray*}
\{ S_{M,T} ; \ T \in \mathcal{M}_{n_1,M} \ and \ M \in \mathcal{P}(\Lambda) \}
\end{eqnarray*}
where $\mathcal{M}_{n_1,M}$ denote the set of trees built on the grid ${ \{ X_i; \ (X_i,Y_i) \in \mathcal{L}_1 \} }$ with splits on the variables in $M$.
\noindent By considering this bigger collection of models,
we no longer have partitions built from an initial one.
So, we use lemma \ref{faux-chi2-2} instead of lemma \ref{faux-chi2}.

\noindent Let us denote, for any $M \in \mathcal{P}(\Lambda)$ and any $T \in \mathcal{M}_{n_1,M}$,
$$ s_{M,T} = \underset{u \in S_{M,T}}{argmin}\  \| s -u \|_{n_1}^2  \ \ {\text{and}} \ \
\varepsilon_{M,T} = \underset{u \in S_{M,T}}{argmin} \  \| \varepsilon-u \|_{n_1}^2 $$

\noindent Following the proof of theorem 1 in (\cite{Birge}), we get
\begin{eqnarray} \label{base1}
(1-\theta)\| s -\tilde{s} \|_{n_1}^2 = \Delta_{\widehat{M,T}} + \underset{(M,T)}{\inf} R_{M,T}
\end{eqnarray}
\noindent where
\begin{eqnarray*}
\Delta_{M,T} &=& (2-\theta) \| \varepsilon_{M,T} \|_{n_1}^2
- 2 < \varepsilon, s-s_{M,T} >_{n_1}
- \theta \| s -s_{M,T} \|_{n_1}^2  - pen(M,T) \\
R_{M,T} &=& \| s -s_{M,T} \|_{n_1}^2 - \| \varepsilon_{M,T} \|_{n_1}^2
+ 2 < \varepsilon, s-s_{M,T} >_{n_1} + pen(M,T)
\end{eqnarray*}

\noindent We are going first to control $\Delta_{\widehat{M,T}}$.
\noindent Let us denote
\begin{enumerate}
\item[$\bullet$] $\delta=5\frac{\rho}{\sigma^2}log\left(\frac{n_1}{p}\right)$
\item[$\bullet$] $\Omega=\{\forall \ 1 \leq i \leq n_1, \vert \varepsilon_i \vert \leq \delta \sigma^2\}$
\end{enumerate}

\noindent Thanks to lemma \ref{faux-chi2-2}, we get that for any $M \in \mathcal{P}(\Lambda)$, $T \in \mathcal{M}_{n_1,M}$ and any $x > 0$
\begin{eqnarray} \label{II1}
\nonumber \mathbb{P} \Bigg(  \| \varepsilon_{M,T} \|_{n_1}^2 \1_{\Omega}
& \geq & \frac{\sigma^2}{n_1} |T| + 4(1+\rho \delta) \frac{\sigma^2}{n_1}\sqrt{2|T|x}
+ 4 \frac{\sigma^2}{n_1} x
\ \Big| \{ X_i; \ (X_i,Y_i) \in \mathcal{L}_1 \} \Bigg) \\
\leq  e^{-x} & &
\end{eqnarray}
\noindent and
\begin{eqnarray*}
\mathbb{P} \left( \Omega^c \
\Big| \{ X_i; \ (X_i,Y_i) \in \mathcal{L}_1 \} \right)
\leq 2 n_1 \exp \left( \frac{- \sigma^2 \delta^2 }{2(1+\rho \delta)} \right)
\end{eqnarray*}

Thanks to assumption (\ref{A}), like to the $(M1)$ case, we get that for any $M \in \mathcal{P}(\Lambda)$,$T \in \mathcal{M}_{n_1,M}$ and any $x>0$,
\begin{eqnarray} \label{II2}
\nonumber \mathbb{P}\Big( -<\varepsilon, s-s_{M,T} >_{n_1} & \geq &
\frac{\sigma}{\sqrt{n_1}} \| s-s_{M,T} \|_{n_1} \sqrt{2x}
+ \frac{2\rho R }{n_1} x
\ \Big| \{ X_i; \ (X_i,Y_i) \in \mathcal{L}_1 \}
\Big) \\
\leq e^{-x} &&
\end{eqnarray}
\noindent Setting $x = x_{M,T} + \xi$ with $\xi >0$ and the weights ${ x_{M,T} = \left(a +(\vert M \vert +1)\left( 1 + log \left( \frac{n_1}{\vert M \vert +1} \right) \right) \right) \vert T \vert +  b \left( 1 + log \left( \frac{p}{\vert M \vert } \right) \right) |M|}$ as defined in lemma \ref{lemme2-poids}, and summing all inequalities (\ref{II1}) and (\ref{II2}) with respect to $M \in \mathcal{P}(\Lambda)$ and $T \in \mathcal{M}_{n_1,M}$, we derive a set $ E_{\xi} $ such that
\begin{itemize}
\item $ \mathbb{P} \left( E_{\xi}^c \ | \{ X_i; \ (X_i,Y_i) \in \mathcal{L}_1 \} \right) \leq  2 e^{-\xi} \Sigma(a,b) $ \\
\item on the set $ E_{\xi} \bigcap \Omega $, for any $(M,T)$,
\begin{eqnarray*}
\Delta_{M,T} & \leq & (2-\theta)\frac{\sigma^2}{n_1}|T| + 4(1+\rho \delta)(2-\theta)\frac{\sigma^2}{n_1}\sqrt{2|T|(x_{M,T}+ \xi)}
+ 2(1+\rho \delta)(2-\theta) \frac{\sigma^2}{n_1} (x_{M,T}+ \xi) \\
& & + 2 \frac{\sigma}{\sqrt{n_1}} \| s-s_{M,T} \|_{n_1} \sqrt{2(x_{M,T}+ \xi)}  + 4 \frac{\rho R}{n_1} (x_{M,T}+ \xi) \\
& & - \theta \| s-s_{M,T} \|_{n_1}^2 - pen(M,T)
\end{eqnarray*}
\end{itemize}
\noindent where $\Sigma(a,b) = \frac{e^{-a}}{1-e^{-a}} \frac{e^{-(b-1)}}{1-e^{-(b-1)}} $.\\

\noindent Using the two following inequalities
\begin{enumerate}
\item[]
$2 \frac{\sigma}{\sqrt{n_1}} \| s-s_{M,T} \|_{n_1} \sqrt{2(x_{M,T}+ \xi)}
\leq \theta \| s-s_{M,T} \|_{n_1}^2 + \frac{2}{\theta} \frac{\sigma^2}{n_1} (x_{M,T}+ \xi)$,
\item[]
$2 \sqrt{|T|(x_{M,T} + \xi)} \leq \eta |T| + \eta^{-1} (x_{M,T}+ \xi)$
\end{enumerate}
with $\eta = \frac{K+\theta-2}{2-\theta} \frac{1}{4\sqrt{2}} >0$, we derive that on the set $ E_{\xi} \bigcap \Omega $, for any $(M,T)$,

\begin{eqnarray*}
\Delta_{M,T} & \leq & (2-\theta)\frac{\sigma^2}{n_1}\vert T \vert + 4\sqrt{2}(1+\rho \delta)(2-\theta)\frac{\sigma^2}{n_1}\sqrt{\vert T \vert(x_{M,T}+\xi)} +\left( 2(1+\rho \delta)(2-\theta)+\frac{2}{\theta}+4\frac{\rho}{\sigma^2}R \right)\frac{\sigma^2}{n_1}(x_{M,T}+\xi) -pen(M,T)  \\
& \leq & K \frac{\sigma^2}{n_1} |T| + \left( 2(1+\rho \delta)(2-\theta) \left(
1 + \frac{4(1+\rho \delta)(2-\theta)}{K+\theta-2} \right)
+ \frac{2}{\theta} + 4 \frac{\rho}{\sigma^2} R  \right)
\frac{\sigma^2}{n_1} (x_{M,T}+ \xi)
- pen(M,T)
\end{eqnarray*}

\noindent Taking a penalty $pen(M,T)$ which compensates for all the other terms in $(M,T)$, i.e.
\begin{eqnarray} \label{pen1}
pen(M,T) \geq  K \frac{\sigma^2}{n_1}|T| +
\left[ 2(1+\rho \delta)(2-\theta) \left( 1 + \frac{4(1+\rho \delta)(2-\theta)}{K+\theta-2} \right)
+ \frac{2}{\theta} + 4 \frac{\rho}{\sigma^2} R  \right]
\frac{\sigma^2}{n_1}x_{M,T} \nonumber \\
&
\end{eqnarray}

\noindent we get that, on the set $ E_{\xi}$
$$\Delta_{\widehat{M,T}} \1_{\Omega}
\leq  \left( 2(1+\rho \delta)(2-\theta) \left(
1 + \frac{4(1+\rho \delta)(2-\theta)}{K+\theta-2} \right)
+ \frac{2}{\theta} + 4 \frac{\rho}{\sigma^2} R  \right)
 \frac{\sigma^2}{n_1} \xi$$

\noindent We are going now to control $ \underset{(M,T)}{\inf} R_{M,T}$.

\noindent In the same way we deduced (\ref{II2}) from assumption (\ref{A}), we get that for any $M \in \mathcal{P}(\Lambda)$ and $T \in \mathcal{M}_{n_1,M}$ and any $x>0$
\begin{eqnarray*}
\mathbb{P}\Big( <\varepsilon, s-s_{M,T} >_{n_1} & \geq &
\frac{\sigma}{\sqrt{n_1}} \| s-s_{M,T} \|_{n_1} \sqrt{2x}
+ \frac{2\rho R}{n_1} x
\ \Big|  \{ X_i; \ (X_i,Y_i) \in \mathcal{L}_1 \}
\Big) \\
\leq e^{-x} &&
\end{eqnarray*}

\noindent Thus we derive a set $F_{\xi}$ such that
\begin{itemize}
\item $ \mathbb{P} \left( F_{\xi}^c \ |  \{ X_i; \ (X_i,Y_i) \in \mathcal{L}_1 \} \right) \leq  e^{-\xi} \Sigma(a,b) $ \\
\item on the set $ F_{\xi} $, for any $(M,T)$,
\begin{eqnarray*}
<\varepsilon, s-s_{M,T} >_{n_1} & \leq &
\frac{\sigma}{\sqrt{n_1}} \| s-s_{M,T} \|_{n_1} \sqrt{2\left(x_{M,T}+\xi \right)}
+ \frac{2\rho R}{n_1}\left(x_{M,T}+\xi \right)
\end{eqnarray*}
\end{itemize}

\noindent It follows from definition of $R_{M,T}$ that on the set $F_{\xi}$, for any $(M,T)$,
\begin{eqnarray*}
R_{M,T}= & \leq & \| s - s_{M,T} \|_{n_1}^2 + 2\frac{\sigma}{\sqrt{n_1}} \|s - s_{M,T}\|_{n_1}\sqrt{2(x_{M,T}+\xi)}+\frac{4\rho R}{n_1}(x_{M,T}+\xi)+pen(M,T) \\
& \leq &  2 \| s -s_{M,T}\|_{n_1}^2+\left( 2+4\frac{\rho}{\sigma^2}R\right)\frac{\sigma^2}{n_1}(x_{M,T}+\xi)+pen(M,T) \\
& \leq & 2 \| s -s_{M,T}\|_{n_1}^2+ 2pen(M,T) + \left( 2+4\frac{\rho}{\sigma^2}R\right)\frac{\sigma^2}{n_1} \xi\\
& &
\end{eqnarray*}

\noindent We conclude that on $E_{\xi} \cap F_{\xi} \cap \Omega$
$$
(1-\theta) \|s-\tilde{s} \|_{n_1}^2 \leq 2\underset{(M,T)}{\inf} \{\|s-s_{M,T}\|_{n_1}^2+pen(M,T)\}+\Upsilon \frac{\sigma^2}{n_1}\xi
$$

\noindent And, for $p\leq log(n_1)$,
\begin{eqnarray*}
\mathbb{P}\left(E_{\xi}^c \cup F_{\xi}^c \cup \Omega^c \right) & \leq & 3e^{-\xi}\Sigma(a,b) \\
& & +\frac{1}{n_1}\underbrace{2exp\left(\frac{-\frac{5\rho^2}{\sigma^2}(log n_1)^2+\frac{50 \rho^2}{\sigma^2}(log n_1)(log log n_1)+4log n_1}{2\left(1+\frac{5\rho^2}{\sigma^2}log n_1 \right)} \right)}_{\epsilon(n_1)}
\end{eqnarray*}

\noindent Finally, we have the following result:\\
\noindent Denoting by
\begin{eqnarray*}
\Upsilon & = & 2(1+\rho \delta)(2-\theta)\left(1+\frac{2(1+\rho \delta)(2-\theta)}{K+\theta-2}\right) +\frac{2}{\theta} \\
 & = & \Big[2(1+5\frac{\rho^2}{\sigma^2}log\left(\frac{n_1}{p}\right))(2-\theta)\left(1+\frac{4(1+5\frac{\rho^2}{\sigma^2}log\left(\frac{n_1}{p}\right))(2-\theta)}{K+\theta-2}\right)+\frac{2}{\theta}\Big]
\end{eqnarray*}

\noindent and
$$
\epsilon(n_1)=2exp\left(\frac{-\frac{5\rho^2}{\sigma\epsilon^2}(log n_1)^2+\frac{50\rho^2}{\sigma^2}(log n_1)(log log n_1)+4log n_1}{2\left( 1 + \frac{5\rho^2}{\sigma^2}log n_1\right)} \right) \underset{n_1 \rightarrow +\infty}{\rightarrow} 0
$$

\noindent Taking a penalty which satisfies: $\forall (M,T)$
$\forall \ M \in \mathcal{P}(\Lambda)$ $ \forall \ T \preceq  T_{max}^{(M)}$

\begin{eqnarray*}
pen(M,T) & \geq & K\frac{\sigma^2}{n_1}\vert T \vert \\
& & +\left( \Upsilon +4\frac{\rho}{\sigma^2}R\right)\frac{\sigma^2}{n_1}\left(a+(\vert M \vert +1)\left(1+log \left( \frac{n_1}{\vert M \vert +1} \right) \right) \right) \vert T \vert \\
& & + \left(\Upsilon+4\frac{\rho}{\sigma^2}R\right) \frac{\sigma^2}{n_1}b\left(1+log \left(\frac{p}{\vert M \vert } \right) \right) \vert M \vert
\end{eqnarray*}

\noindent we have $\forall \xi >0$, with probability $\geq 1-3e^{-\xi}\Sigma(a,b)-\frac{1}{n_1}\epsilon(n_1)$
$$
(1-\theta) \| s -\tilde{s} \|_{n_1}^2  \leq  2 \underset{(M,T)}{\inf} \left\{\| s - s_{M,T} |_{n_1}^2  +  pen(M,T) \right\} +\left(\Upsilon+2+8\frac{\rho}{\sigma^2}R\right)\frac{\sigma^2}{n_1}\xi
$$

\noindent We deduce the proposition by noticing that
\begin{eqnarray*}
\Upsilon & \leq & 8(2-\theta)max\left\{1;\frac{4(2-\theta)}{K+\theta-2}\right\}\left(1+25\frac{\rho^4}{\sigma^4}log^2\left(\frac{n_1}{p}\right)\right)+\frac{2}{\theta} \\
& \leq & C(K,\theta)\left(1+\frac{\rho^4}{\sigma^4}log^2\left(\frac{n_1}{p}\right) \right)
\end{eqnarray*}

\noindent and by taking $K=2$, $\theta \rightarrow 1$, $a\rightarrow 0$ and $b\rightarrow 1$.
\hfill $\Box$\\

\subsubsection{Proof of the proposition \ref{propF}: }

\noindent It follows from the definition of $ \tilde{\tilde{s}} $ that
for any $\tilde{s}(\alpha,\beta) \in \mathcal{G}$
\begin{eqnarray} \label{bf}
\left\| s - \tilde{\tilde{s}} \right\|_{n_3}^2 \leq
\left\| s - \tilde{s}(\alpha,\beta) \right\|_{n_3}^2
+ 2  \left< \varepsilon, \tilde{\tilde{s}} - \tilde{s}(\alpha,\beta) \right>
_{n_3}
\end{eqnarray}

\noindent Denoting $ M_{\alpha, \beta, \alpha^{\prime}, \beta^{\prime}} =
max \left\{ \left|
\tilde{s}(\alpha^{\prime},\beta^{\prime})(X_i)
- \tilde{s}(\alpha,\beta)(X_i)
\right|; \ (X_i,Y_i) \in \mathcal{L}_3 \right\}   $, we deduce from assumption (\ref{A})
that for any $\tilde{s}(\alpha,\beta)$ and $\tilde{s}(\alpha^{\prime},\beta^{\prime}) \in \mathcal{G}$
\begin{eqnarray*}
& log  \mathbb{E}\left[exp(\lambda < \varepsilon, \tilde{s}(\alpha,\beta) >_{n_3}) | \mathcal{L}_1, \mathcal{L}_2 \text{ and } \{X_i; \ (X_i,Y_i) \in \mathcal{L}_3\} \right] \\
& \leq \frac{\sigma^2 \|\tilde{s}(\alpha^{\prime},\beta^{\prime})-\tilde{s}(\alpha,\beta)\|_{n_3}^2 \lambda^2}{2n_3\left(1-\frac{\rho}{n_3}M_{\alpha,\beta,\alpha^{\prime},\beta^{\prime}} |\lambda|\right)} \text{ if } |\lambda| < \frac{n_3}{\rho M_{\alpha,\beta,\alpha^{\prime},\beta^{\prime}}}
\end{eqnarray*}

\noindent Thus we get that for any $\tilde{s}(\alpha,\beta)$, $\tilde{s}(\alpha^{\prime},\beta^{\prime}) \in \mathcal{G}$ and $x>0$
\begin{eqnarray*}
\mathbb{P} & \left( \left< \varepsilon, \tilde{s}(\alpha^{\prime},\beta^{\prime})
- \tilde{s}(\alpha,\beta) \right>_{n_3} \geq
\frac{\sigma}{\sqrt{n_3}} \| \tilde{s}(\alpha^{\prime},\beta^{\prime})
- \tilde{s}(\alpha,\beta) \|_{n_3} \sqrt{2x} +
M_{\alpha, \beta, \alpha^{\prime}, \beta^{\prime}} \frac{\rho}{n_3} x \right. \\
& \left. \Big| \ \mathcal{L}_1, \ \mathcal{L}_2, \left\{ X_i, \ (X_i,Y_i) \in \mathcal{L}_3 \right\} \right) \leq e^{-x}
\end{eqnarray*}

\noindent Setting $x = 2 \log \mathcal{K} + \xi$ with $\xi >0$, and summing all these inequalities
with respect to $\tilde{s}(\alpha,\beta)$ and
$\tilde{s}(\alpha^{\prime},\beta^{\prime}) \in \mathcal{G}$ , we derive a set $ E_{\xi} $ such that
\begin{itemize}
\item $ \mathbb{P} \left( E_{\xi}^c \ | \mathcal{L}_1, \ \mathcal{L}_2,
\text{ and } \{ X_i; \ (X_i,Y_i) \in \mathcal{L}_3 \} \right) \leq  e^{-\xi} $ \\
\item on the set $ E_{\xi}$, for any $\tilde{s}(\alpha,\beta)$ and
$\tilde{s}(\alpha^{\prime},\beta^{\prime}) \in \mathcal{G}$
\begin{eqnarray*}
\left< \varepsilon, \tilde{s}(\alpha^{\prime},\beta^{\prime})
- \tilde{s}(\alpha,\beta) \right>_{n_3} & \leq &
\frac{\sigma}{\sqrt{n_3}} \| \tilde{s}(\alpha^{\prime},\beta^{\prime})
- \tilde{s}(\alpha,\beta) \|_{n_3} \sqrt{2(2 \log \mathcal{K} + \xi)} \\
 & & + M_{\alpha, \beta, \alpha^{\prime}, \beta^{\prime}} \frac{\rho}{n_3} (2 \log \mathcal{K} + \xi)
\end{eqnarray*}
\end{itemize}

\noindent It remains to control ${ M_{\alpha, \beta, \alpha^{\prime}, \beta^{\prime}} }$ in the two situations $(M1)$ and $(M2)$ (except if $\rho = 0$).

\noindent In the $(M1)$ situation, we consider the set
\begin{eqnarray*}
\Omega_1 = \underset{M \in \mathcal{P}(\Lambda)}{\bigcap}
\left\{ \forall t \in \widetilde{T_{max}^{(M)}} \left|
\sum_{\underset{X_i \in t }{(X_i,Y_i) \in \mathcal{L}_2}} \varepsilon_i
\right|
\leq R \left|
\left\{i; \ (X_i,Y_i) \in \mathcal{L}_2 \text{ and } X_i \in t  \right\}
\right| \right\}
\end{eqnarray*}

\noindent Thanks to assumption (\ref{A}), we get that for any $\lambda \in (-1/\rho,\ 1/\rho)$
\begin{eqnarray*}
& log \mathbb{E}\Big[ exp\left(\lambda \sum_{(X_i,Y_i)\in \mathcal{L}_2, X_i \in t} \varepsilon_i \right) \Big| \mathcal{L}_1 {\text{ and }} \{X_i; \ (X_i,Y_i) \in \mathcal{L}_2\} \Big] \\
& \leq \frac{\lambda^2 \sigma^2}{2(1-\rho|\lambda|)}|\{i; \ (X_i,Y_i) \in \mathcal{L}_2 {\text{ and }} X_i \in t \}|
\end{eqnarray*}

\noindent It follows that for any $x>0$
\begin{eqnarray*}
\mathbb{P}\left( \left| \sum_{ \underset{X_i \in t}{(X_i,Y_i) \in
\mathcal{L}_2} } \varepsilon_i \right| \geq x \Bigg| \mathcal{L}_1
\text{ and } \{ X_i; \ (X_i,Y_i) \in \mathcal{L}_2 \}
\right) & \leq &
2 e^{\frac{-x^2}{2\left( \sigma^2\left|
\{i; \ (X_i,Y_i) \in \mathcal{L}_2 \text{ and } X_i \in t \} \right|
+ \rho x \right)}}
\end{eqnarray*}

\noindent Taking $x = R \left|
\left\{i; \ (X_i,Y_i) \in \mathcal{L}_2 \text{ and } X_i \in t  \right\}
\right| $ and summing all these inequalities, we get that
\begin{eqnarray*}
\mathbb{P} \left( \Omega_1^c \Big| \  \mathcal{L}_1 \text{ and }
\{ X_i; \ (X_i,Y_i) \in \mathcal{L}_2 \} \right)
& \leq & 2^{p+1} \frac{n_1}{N_{min}} \exp \left( \frac{-R^2 N_{min}}
{2(\sigma^2 + \rho R)} \right)
\end{eqnarray*}

\noindent On the set $\Omega_1$, as for any $(M,T)$,
$\| \hat{s}_{M,T} \|_{\infty} \leq 2R$, we have
${ M_{\alpha, \beta, \alpha^{\prime}, \beta^{\prime}} \leq 4R }$.

\noindent Thus, on the set $\Omega_1 \bigcap E_{\xi}$,
for any $\tilde{s}(\alpha,\beta) \in \mathcal{G}$
\begin{eqnarray*}
\left< \varepsilon, \tilde{\tilde{s}}
- \tilde{s}(\alpha,\beta) \right>_{n_3} & \leq &
\frac{\sigma}{\sqrt{n_3}} \| \tilde{\tilde{s}}
- \tilde{s}(\alpha,\beta) \|_{n_3} \sqrt{2(2 \log \mathcal{K} + \xi)} +
4R \frac{\rho}{n_3} (2 \log \mathcal{K} + \xi)
\end{eqnarray*}

It follows from (\ref{bf}) that, on the set $\Omega_1 \bigcap E_{\xi}$,
for any $\tilde{s}(\alpha,\beta) \in \mathcal{G}$ and
any $\eta \in (0;1)$
\begin{eqnarray*}
\|s-\tilde{\tilde{s}}\|_{n_3}^2 \leq \|s-\tilde{s}(\alpha,\beta)\|{n_3}^2+(1-\eta)\|\tilde{\tilde{s}}-\tilde{s}(\alpha,\beta)\|{n_3}^2+ \frac{2}{1-\eta}\frac{\sigma^2}{n_3}(2log K +\xi) +\frac{8\rho R}{n_3}(2log K + \xi)
\end{eqnarray*}

\noindent and
\begin{eqnarray*}
\eta^2 \left\| s - \tilde{\tilde{s}} \right\|_{n_3}^2 \leq
(1+\eta^{-1}-\eta) \left\| s - \tilde{s}(\alpha,\beta) \right\|_{n_3}^2
+ \left( \frac{2}{1-\eta}\sigma^2 + 8 \rho R \right)
\frac{(2 \log \mathcal{K} + \xi)}{n_3}
\end{eqnarray*}

\noindent Taking $p\leq \log n_2$ and $N_{min} \geq
4 \frac{\sigma^2 + \rho R}{R^2} \log n_2$, we have
$$ \mathbb{P} \left( \Omega_1^c \right) \leq
\frac{R^2}{2(\sigma^2+ \rho R)} \frac{1}{n_2^{1-\log 2}} $$

\noindent Finally, in the $(M1)$ situation, we have\\
for any $\xi >0$, with probability
$\geq 1 - e^{-\xi} - \frac{R^2}{2(\sigma^2+ \rho R)} \frac{1}{n_2^{1-\log 2}}$,
$\forall \eta \in (0,1)$,
\begin{eqnarray*}
\left\| s - \tilde{\tilde{s}} \right\|_{n_3}^2 \leq
\frac{(1+\eta^{-1}-\eta)}{\eta^2}
\underset{\tilde{s}(\alpha,\beta) \in \mathcal{G}}{\inf}
\left\| s - \tilde{s}(\alpha,\beta) \right\|_{n_3}^2
+ \frac{1}{\eta^2} \left( \frac{2}{1-\eta}\sigma^2 + 8 \rho R \right)
\frac{(2 \log \mathcal{K} + \xi)}{n_3}
\end{eqnarray*}

\noindent In the $(M2)$ situation, we consider the set
\begin{eqnarray*}
\Omega_2 =
\left\{ \forall 1\leq i \leq n_1 \left|  \varepsilon_i \right|
\leq 3 \rho \log n_1 \right\}
\end{eqnarray*}

\noindent Thanks to assumption (\ref{A}), we get that
\begin{eqnarray*}
\mathbb{P} \left( \Omega_2^c \Big| \
\{ X_i; \ (X_i,Y_i) \in \mathcal{L}_1 \} \right)
& \leq & 2 n_1 \exp \left( - \frac{9 \rho^2 \log^2 n_1}{2(
\sigma^2 + 3 \rho^2 \log n_1)} \right)
\end{eqnarray*}

\noindent with $\epsilon(n_1) =
2 n_1 \exp \left( - \frac{9 \rho^2 \log^2 n_1}{2(
\sigma^2 + 3 \rho^2 \log n_1)} \right)
\underset{n_1 \rightarrow + \infty}{\longrightarrow} 0$

\noindent On the set $\Omega_2$, as for any $(M,T)$,
$\| \hat{s}_{M,T} \|_{\infty} \leq R+ 3 \rho \log n_1$, we have
${ M_{\alpha, \beta, \alpha^{\prime}, \beta^{\prime}} \leq
2(R+ 3 \rho \log n_1) }$.\\
\noindent  Thus, on the set $\Omega_2 \bigcap E_{\xi}$,
for any $\tilde{s}(\alpha,\beta) \in \mathcal{G}$
\begin{eqnarray*}
\left< \varepsilon, \tilde{\tilde{s}}
- \tilde{s}(\alpha,\beta) \right>_{n_3} & \leq &
\frac{\sigma}{\sqrt{n_3}} \| \tilde{\tilde{s}}
- \tilde{s}(\alpha,\beta) \|_{n_3} \sqrt{2(2 \log \mathcal{K} + \xi)} +
2(R+ 3 \rho \log n_1) \frac{\rho}{n_3} (2 \log \mathcal{K} + \xi)
\end{eqnarray*}

\noindent It follows from (\ref{bf}) that, on the set $\Omega_2 \bigcap E_{\xi}$,
for any $\tilde{s}(\alpha,\beta) \in \mathcal{G}$ and
any $\eta \in (0;1)$
\begin{eqnarray*}
\|s-\tilde{\tilde{s}}\|_{n_3}^2 & \leq &  \|s-\tilde{s}(\alpha,\beta)\|_{n_3}^2 + (1-\eta)\|\tilde{\tilde{s}}-\tilde{s}(\alpha,\beta)\|_{n_3}^2+\frac{2}{1-\eta}\frac{\sigma^2}{n_3}(2log K+\xi) \\
& & + \frac{4\rho(R+3\rho log n_1)}{n_3}(2log K + \xi)
\end{eqnarray*}

\noindent and
\begin{eqnarray*}
\eta^2 \left\| s - \tilde{\tilde{s}} \right\|_{n_3}^2 \leq
(1+\eta^{-1}-\eta) \left\| s - \tilde{s}(\alpha,\beta) \right\|_{n_3}^2
+ \left( \frac{2}{1-\eta}\sigma^2 +  4 \rho (R+ 3 \rho \log n_1)  \right)
\frac{(2 \log \mathcal{K} + \xi)}{n_3}
\end{eqnarray*}

\noindent Finally, in the $(M2)$ situation, we have for any $\xi >0$, with probability
$\geq 1 - e^{-\xi} - \epsilon(n_1)$,
$\forall \eta \in (0,1)$,
\begin{eqnarray*}
\left\| s - \tilde{\tilde{s}} \right\|_{n_3}^2 & \leq &
\frac{(1+\eta^{-1}-\eta)}{\eta^2}
\underset{\tilde{s}(\alpha,\beta) \in \mathcal{G}}{\inf}
\left\| s - \tilde{s}(\alpha,\beta) \right\|_{n_3}^2 \\
& & + \frac{1}{\eta^2} \left( \frac{2}{1-\eta}\sigma^2 + 4 \rho R
+ 12 \rho^2 \log n_1 \right)
\frac{(2 \log \mathcal{K} + \xi)}{n_3}
\end{eqnarray*}
\hfill $\Box$

\bibliographystyle{plainnat}
\bibliography{bibmc-jmva2}

\end{document}